\documentclass[12pt,a4paper]{article}
\usepackage{amsfonts}
\usepackage{mathrsfs}
\usepackage{color}
\usepackage{stmaryrd}
\usepackage{amsmath}
\usepackage{amssymb}
\usepackage{bbm}
\usepackage{graphicx}
\usepackage{enumerate}
\usepackage{theorem}
\usepackage{hyperref}
\makeatletter
\def\tank#1{\protected@xdef\@thanks{\@thanks
        \protect\footnotetext[0]{#1}}}
\def\bigfoot{

    \@footnotetext}
\makeatother

\newcommand{\ea}{\end{array}}
\newtheorem{theorem}{Theorem}[section]
\newtheorem{proposition}{Proposition}[section]
\newtheorem{claim}{Claim}[section]

\newtheorem{lemma}{Lemma}[section]
\newtheorem{definition}{Definition}[section]
\newtheorem{remark}{Remark}[section]
\newtheorem{Condition}{Condition}[section]

{\theorembodyfont{\rmfamily}
}

\newenvironment{proof}{Proof.}

\def \eqref#1{\hbox{(\ref{#1})}}

\allowdisplaybreaks

\begin{document}
\title{\Large \bf Large deviations for stochastic generalized porous media equations driven by L\'{e}vy noise \thanks{This work is partially supported by National Key R\&D
Program of China(No. 2022YFA1006001). Weina Wu's research is supported by the National Natural Science Foundation of China (NSFC) (No. 11901285), China Scholarship Council (CSC) (No. 202008320239) and DFG through CRC 1283. Jianliang Zhai's research is supported by NSFC (No. 12131019, 11971456, 11721101),  and the Fundamental Research Funds for the Central Universities (No. WK3470000031).} }

\author{{Weina Wu}$^{a,c}$\footnote{E-mail:wuweinaforever@163.com}~~~ {Jianliang Zhai}$^b$\footnote{E-mail:zhaijl@ustc.edu.cn}
\\
 \small  a. School of Economics, Nanjing University of Finance and Economics,\\
 \small Nanjing, Jiangsu 210023, China.\\
 \small  b. School of Mathematical Sciences,
University of Science and Technology of China,\\
\small Hefei, Anhui 230026, China.\\
\small c. Faculty of Mathematics, University of Bielefeld,\\
\small D-33615 Bielefeld, Germany.}\,
\date{}
\maketitle

\begin{center}
\begin{minipage}{130mm}
{\bf Abstract.} We establish a large deviation principle (LDP) for a class of stochastic porous media equations driven by L\'{e}vy-type noise on a $\sigma$-finite measure space $(E,\mathcal{B}(E),\mu)$, with the Laplacian replaced by a negative definite self-adjoint operator. One of the main contributions of this paper is that we do not assume the compactness of embeddings in the corresponding Gelfand triple, and to compensate for this generalization, a new procedure is provided. This is the first paper to deal with LDPs for stochastic evolution equations with L\'evy noise without compactness conditions. The coefficient $\Psi$ is assumed to satisfy nondecreasing Lipschitz nonlinearity, so an important physical problem covered by this case is the Stefan problem. Numerous examples of negative definite self-adjoint operators are applicable to our results, for example, for open $E\subset\Bbb{R}^d$, $L=$ Laplacian or fractional Laplacians, i.e., $L=-(-\Delta)^\alpha,\ \alpha\in(0,1]$, generalized Schr\"{o}dinger operators, i.e., $L=\Delta+2\frac{\nabla \rho}{\rho}\cdot\nabla$, Laplacians on fractals is also included.

\vspace{3mm} {\bf Keywords.} stochastic porous media equations; L\'{e}vy noise; large deviation principle; weak convergence; sub-Markovian; strongly continuous contraction semigroup.
\end{minipage}
\end{center}

\section{Introduction}
\setcounter{equation}{0}
 \setcounter{definition}{0}

Let $(E,\mathcal{B}(E),\mu)$ be a $\sigma$-finite measurable space. We assume that $(E, \mathcal{B}(E))$ is a standard measurable space, i.e., $\sigma$-isomorphic to a Polish space; see \cite{P67}. Let $L^2(\mu):=L^2(E,\mathcal{B}(E),\mu)$.
The purpose of this paper is to establish a Freidlin Wentzell-type large deviation principle (LDP) to the following stochastic generalized porous media equations (SGPMEs) driven by L\'{e}vy process:
\begin{equation} \label{eq:1}
\left\{ \begin{aligned}
&dX^\epsilon(t)=L\Psi(X^\epsilon(t))dt+\epsilon\int_{Z}f(t,X^\epsilon(t-),z)\widetilde{N}^{\epsilon^{-1}}(dz,dt),\ t\in[0,T],\\
&X^\epsilon(0)= x\in L^2(\mu),
\end{aligned} \right.
\end{equation}
where $L$ is the infinitesimal generator of a symmetric sub-Markovian strongly continuous contraction semigroup $(P_t)_{t\geq0}$ on $L^2(\mu)$.
$\Psi(\cdot):\Bbb{R}\rightarrow\Bbb{R}$ is a monotonically nondecreasing Lipschitz continuous function. $\epsilon>0$ is a small parameter, and $\widetilde{N}^{\epsilon^{-1}}$ is a compensated Poisson random measure on $[0,T]\times Z$ with a $\sigma$-finite intensity measure $\epsilon^{-1}\lambda_T\otimes\nu$, where $\lambda_T$ is the Lebesgue measure on $[0,T]$, and $\nu$ is a $\sigma$-finite measure on $Z$.
For the definition of the compensated Poisson random measure $\widetilde{N}^{\epsilon^{-1}}$, see Section \ref{Section2} below. For the precise conditions on $\Psi$ and $f$, see Section \ref{Section3} below.

The classical porous media equation:
\begin{eqnarray}\label{classical}
dX(t)=\Delta X^m(t)dt
\end{eqnarray}
on a domain in $\Bbb{R}^d$, for $m>1$, models the flow of ionized gases at high temperature, nonlinear heat transfer, and filtration of incompressible fluids through a porous stratum; see, e.g., \cite{A, ZR} and references therein. Since the foundational work in \cite{DR, DR04}, there have been many publications on the study of stochastic porous media equations (SPMEs), including the well-posedness of solutions and their long-time behaviors; see, e.g., \cite{BDPR04, BGLR, BRR, DR, DR04, G13, G14} and references therein.

SGPMEs extend the study of SPMEs with a nonlinear term $\Delta X^m(t)$ on a domain in $\Bbb{R}^d$ to that with a nonlinear term $L\Psi(X(t))$ on general measure spaces $(E,\mathcal{B}(E),\mu)$.
Typically, the methods and techniques available for investigating SPMEs are unsuitable for SGPMEs, so new and sophisticated tools are needed. For the cases driven by Wiener processes, SGPMEs have been investigated by many people, and papers have been published on the existence and uniqueness of solutions (\cite{DRRW06, RRW, RW, RWX}), LDPs (\cite{RWW, WZ}), invariant measures (\cite{RW07}), the Harnack inequality (\cite{W}), and many other SGPME aspects.
In contrast, only a few studies address SGPMEs driven by L\'{e}vy-type or Poisson-type perturbations, let alone the deriving properties of their solutions. Assuming $(E,\mathcal{B}(E),\mu)$ is a separable probability space, Hou and Zhou \cite{ZH} considered the existence and uniqueness of solutions to SGPMEs driven by L\'{e}vy noise. Later, the ergodicity and exponential stability of the same equation were obtained in \cite{ZH1} and \cite{GZ}, respectively.
In a recent paper \cite{WZ21}, the current paper's authors proved the existence and uniqueness of solutions to \eqref{eq:1} when $(E,\mathcal{B}(E),\mu)$ is a $\sigma$-finite and standard measurable space. In order to have a better understanding of the asymptotic behavior of the solution to \eqref{eq:1}, the present paper investigates an LDP for \eqref{eq:1}.  No LDP results have previously been published for SGPMEs driven by L\'{e}vy-type noise.

Our assumptions of $(E,\mathcal{B}(E),\mu)$ and $L$ are the same as in \cite{RWX, WZ, WZ21}, and $(L,D(L))$ is a Dirichlet operator on $L^2(\mu)$ (\cite{MR}). These assumptions mean that once one proves that  $L$ is a generator on an $L^2$-space of a Dirichlet form, both the results of the current paper and \cite{RWX, WZ, WZ21} apply to the $L$.
For example, let $E:=U\subset \Bbb{R}^d$, $U$ be open, and $\mu$ a positive Radon measure on $U$ such that $supp[\mu]=U$.
Using the Dirichlet form theory, one can define its associated Dirichlet operator $L$ on $L^2(U,\mu)$.
If $L$ is the Friedrichs extension of the operator $L_0=\Delta+2\frac{\nabla \rho}{\rho}\cdot\nabla$ on $L^2(\Bbb{R}^d, \rho^2dx)$, where $dx$ denotes the Lebesgue measure and $\rho\in H^1(\Bbb{R}^d)$, one can prove that there exists a Dirichlet form on $L^2(\Bbb{R}^d, \rho^2dx)$ and $L$ is the corresponding Dirichlet operator.
For explicit verifications of the above claims about $E$ and $L$, we refer the reader to \cite[Section 4]{RWX}.
Other interesting results have been proven,
such as, if $E=\Bbb{R}^d$ and $L=-(-\Delta)^\alpha, \alpha\in(0,1]$, no restriction on $d$ is needed;
if $E$ is a fractal, we can take $L$ to be the Laplace operator on this fractal;
and the results also apply to
generalized $\rm Schr\ddot{o}dinger$ operators (i.e., $L=\Delta+2\frac{\nabla \rho}{\rho}\cdot\nabla$).
We refer to \cite[Chapter II]{MR} for many other examples of related results.

The coefficient $\Psi$ is assumed to satisfy the property of nondecreasing Lipschitz nonlinearity (see \textbf{(H1)} in Section \ref{Section3}), and an important physical problem covered by this case is the Stefan problem (\cite[Section 1.1.1]{BDR}). To be more precise, if $L=\Delta$ on a bounded open subset of $\Bbb{R}^d$ with Dirichlet boundary conditions, and we take $\Psi$ to be
\begin{eqnarray*}
\Psi(r)=\left\{
          \begin{array}{ll}
            ar, &\text{for}~r<0, \\
            0, & \text{for}~0\leq r\leq\rho, \\
            b(r-\rho), & \text{for}~r>\rho,
          \end{array}
        \right.
\end{eqnarray*}
where $a,b,\rho>0$, then $\Psi$ fulfills \textbf{(H1)} with Lipschitz constant $\max\{a,b\}$. In this case, \eqref{eq:1} reduces to the classical two-phase Stefan problem, describing the situation where the melting (or solidification) of phase changing materials in the presence of L\'evy noise.
Here, $\Psi$ is the inverse of the enthalpy function associated with the phase transition, and $X^\epsilon$ is related to the temperature $\vartheta$ by the transformation $\vartheta=\Psi(X^\epsilon)$.

The weak convergence approach introduced by \cite{BDM2,BCD} has been applied to study the LDPs in various dynamical systems driven by L\'evy processes; see, e.g., \cite{BPZ,DXZZ 2017, LSZZ, RZ, XZ,YZZ 2015,ZZ 2015}.
Recently, a sufficient condition to verify this large deviation criteria has been improved in the paper \cite{LSZZ} by the second author and his collaborators.
The first sufficient condition was introduced by Matoussi, Sabbagh, and Zhang in \cite{MSZ} for the Wiener case.
The improved sufficient condition seems to be more effective and suitable to deal with SPDEs with highly nonlinear terms; see, e.g., \cite{DWZZ, MSZ, WZ} for the Wiener case.
In this paper, we adopt this improved sufficient condition. The main point of our procedures is to prove the convergence of the so-called skeleton equations; see Subsection \ref{Subsection5.1}. Before this, we need to obtain results on existence and uniqueness and also provide some a priori estimates for solutions to the skeleton equations; see Section \ref{Section4}.

One of the main contributions of this paper is that we do not assume the compactness of embeddings in the corresponding Gelfand triple.
Avoiding this assumption allows us to cover the important class of
models on general $\sigma$-finite measurable spaces $(E,\mathcal{B}(E),\mu)$, and in particular, the models on unbounded domains in $\Bbb{R}^d$.
In sharp contrast to our work, previous results on LDPs for stochastic evolution equations (SEEs) with L\'evy noise heavily rely on the compactness condition for the Gelfand triples because it is crucial in their methods to prove the compactness of solutions to the so-called stochastic control equations and further prove convergence to the solutions of the corresponding skeleton equations.
Hence, in our present work, a different approach must be employed.
The main difficulty lies in the convergence of the skeleton equations, and to overcome this difficulty, we adopt a series of technical methods, including time discretization, a cut-off argument, and relative entropy estimates of a sequence of probability measures; see Subsection \ref{Subsection5.1}. {These successfully refrain the assumption that the coefficient of the noise is H\"{o}lder continuous w.r.t. time in the Wiener case (cf. \cite[page:10006, (H3)]{WZ}), which can be seen as a ``compensation" for dropping the compactness condition of the Gelfand triples.
We want to emphasize the fact that the time discretization approach, which was inspired by the work \cite{CM} and has been widely used in studying LDPs for SEEs driven by Gaussian noise, is also applicable here.
The current paper is the first to deal with LDPs for SEEs with L\'evy noise without the previous compactness conditions.


Finally, we would like to refer the reader to \cite{BLZ, LR, P, PR, PZ} for more background information and results on SPDEs, to \cite{A, BDR} for background on SPMEs, and to \cite{RRW, RW, RWW, RWX, RWX1, RWZ} and the references therein for comprehensive theories of SGPMEs.

The structure of this paper is as follows: In Section \ref{Section2}, we review basic notation about Poisson random measures and introduce the Gelfand triple used throughout this work. Section \ref{Section3} states the precise hypotheses and the main result: the large deviations for \eqref{eq:1}. Section \ref{Section4} is devoted to proving the existence and uniqueness of solutions to the skeleton equations. In Section \ref{Section5}, we prove the main result.

\section{Preliminaries}\label{Section2}
\setcounter{equation}{0}
 \setcounter{definition}{0}

 In this section, we introduce notation, the definition of the compensated Poisson random measure $\widetilde{N}^{\epsilon^{-1}}$ in (\ref{eq:1}), and the Gelfand triple used in this paper.

\subsection{Notation}

For a locally compact Polish space $\mathcal{K}$, let $\mathcal{M}_{FC}(\mathcal{K})$ denote the space of all nonnegative measures $\gamma$ on $({\mathcal{K}},\mathcal{B}({\mathcal{K}}))$ such that $\gamma(K)<\infty$ for every compact $K$ in ${\mathcal{K}}$.
Let $C_c({\mathcal{K}})$ denote the space of continuous functions with compact support, and endow $\mathcal{M}_{FC}({\mathcal{K}})$ with the weakest topology such that for every $\hbar\in C_c({\mathcal{K}})$, the function $\gamma\ni\mathcal{M}_{FC}({\mathcal{K}})\mapsto\langle \hbar,\gamma\rangle:=\int_{\mathcal{K}}\hbar(u)\gamma (du)$ is continuous. This topology can be metrized such that $\mathcal{M}_{FC}({\mathcal{K}})$ is a Polish space; see Section 2 of \cite{BDM2} for more details. Throughout the paper, we use this topology on $\mathcal{M}_{FC}(\mathcal{K})$.

In this paper, we fix $T\in(0,\infty)$. Denote by $\lambda_T$ the Lebesgue measure on $[0,T]$, and $\lambda_\infty$ the Lebesgue measure on $[0,\infty)$.

For a metric space $\Bbb{H}$, the Borel $\sigma$-field on $\Bbb{H}$ is denoted by $\mathcal{B}(\Bbb{H})$.
Let $L^2([0,T]\times\Omega;\Bbb{H})$ denote
the space of all $\Bbb{H}$-valued square-integrable functions on
$[0,T]\times\Omega$,
$L^\infty([0,T],\Bbb{H})$ denote the space of all $\Bbb{H}$-valued uniformly bounded measurable functions on $[0,T]$, $C([0,T];
\Bbb{H})$ denote the space of all $\Bbb{H}$-valued continuous functions on $[0,T]$ equipped with the topology of uniform convergence, and $D([0,T];\Bbb{H})$ denote the space of all $\Bbb{H}$-valued c\`{a}dl\`{a}g functions on $[0,T]$ equipped
with the usual Skorohod topology. For simplicity, the positive constants $c$, $C$, $C_k$, $k=1,2,\ldots$, used in this paper may change from line to line.

\subsection{Poisson random measure}

Let $Z$ be a locally compact Polish space, and $\nu$ be a given $\sigma$-finite positive measure on $(Z,\mathcal{B}(Z))$ with $\nu\in\mathcal{M}_{FC}(Z)$.
Let $Z_T=[0,T]\times Z$, $Y=Z\times [0,\infty)$, and $Y_T=[0,T]\times Y$.

For simplicity, from now on, we write ${\Bbb{M}}:=\mathcal{M}_{FC}(Y_T)$.  Denote by ${\Bbb{P}}$ the unique probability measure on $( {\Bbb{M}},\mathcal{B}( {\Bbb{M})})$, under which the canonical map, $\bar{N}: {\Bbb{M}}\mapsto {\Bbb{M}}$, $ \bar{N}( {m}):=  {m}$, is a Poisson random measure with intensity measure $\bar{\nu}_T=\lambda_T\otimes\nu\otimes\lambda_\infty$.
The corresponding compensated Poisson random measure is denoted by $\widetilde{\bar{N}}$.
Let $\bar{\mathcal{F}_t}:=\sigma\{\bar{N}((0,s]\times A):0\leq s\leq t, A\in\mathcal{B}(Y)\}$, and let $ {\mathcal{F}}_t$ denote the completion under $ {\Bbb{P}}$.
 Denote by $ {\mathcal{P}}$ the predictable $\sigma$-field on $[0,T]\times {\Bbb{M}}$ with the filtration $\{ {\mathcal{F}}_t:0\leq t\leq T\}$ on $( {\Bbb{M}},\mathcal{B}( {\Bbb{M}}))$.
Let $ {\mathcal{A}}$ be the class of all $( {\mathcal{P}}\otimes\mathcal{B}(Z))/\mathcal[0,\infty)$-measurable maps $\varphi:Z_T\times {\Bbb{M}}\rightarrow[0,\infty)$.
For $\varphi\in {\mathcal{A}}$, define a counting process $N^\varphi$ on $Z_T$ by
\begin{eqnarray}
N^\varphi\big((0,t]\times U\big)=\int_{(0,t]\times U\times(0,\infty)}1_{[0,\varphi(s,z)]}(r)\bar{N}(dr,dz,ds), \  \text{for}\ t\in[0,T]\ \text{and}\ U\in\mathcal{B}(Z).
\end{eqnarray}
This $N^\varphi$ is called a controlled random measure, with $\varphi$ selecting the intensity for the points at location $z$ and time $s$, in a possibly random but non-anticipating way.
Analogously, we define a process
\begin{eqnarray}
\widetilde{N}^\varphi\big((0,t]\times U\big)=\int_{(0,t]\times U\times(0,\infty)}1_{[0,\varphi(s,z)]}(r)\widetilde{\bar{N}}(dr,dz,ds).
\end{eqnarray}
When $\varphi(s,z)\equiv\epsilon^{-1}\in(0,\infty)$, we write $N^\varphi=N^{\epsilon^{-1}}$ and
$\widetilde{N}^\varphi=\widetilde{N}^{\epsilon^{-1}}$. Set $\nu_T:=\lambda_T\otimes\nu$. Note that, with respect to $ {\Bbb{P}}$, $N^{\epsilon^{-1}}$ is a Poisson random measure on $Z_T$ with intensity measure $\epsilon^{-1}\nu_T$, and $\widetilde{N}^{\epsilon^{-1}}$ is the compensated Poisson random measure.

Set $(\Omega,\mathcal{F})=(\Bbb{M},\mathcal{B}(\Bbb{M}))$. In the present paper, we study \eqref{eq:1} on the given probability space $(\Omega,\mathcal{F}, \{\mathcal{F}_t,t\in[0,T]\},\Bbb{P})$. Denote the expectation with respect to $\Bbb{P}$ by $\Bbb{E}$.

\subsection{Gelfand triple}
Let $(E,\mathcal{B}(E),\mu)$ be a $\sigma$-finite measurable space. Let $(P_t)_{t>0}$ be a strongly continuous, symmetric sub-Markovian
contraction semigroup on $L^2(\mu)$ with generator $(L, D(L))$. The $\Gamma$-transform of $(P_t)_{t>0}$ is defined by the following Bochner integral
\begin{eqnarray*}
V_ru:=\frac{1}{\Gamma(\frac{r}{2})}\int_0^\infty t^{\frac{r}{2}-1}e^{-t}P_tudt,~~u\in L^2(\mu),~~r>0.
\end{eqnarray*}
In this paper, we consider the Hilbert space $(F_{1,2},\|\cdot\|_{F_{1,2}})$ defined by
$$
F_{1,2}:=V_1(L^2(\mu)),~\text{with~norm}~\|f\|_{F_{1,2}}=|u|_2,~~\text{for}~~f=V_1u,~~ u\in L^2(\mu),
$$
where the norm $|\cdot|_2$ is defined as $|u|_2=(\int_E |u|^2d\mu)^{\frac{1}{2}}$, and its inner product is denoted by $\langle \cdot, \cdot\rangle_2$.
From \cite{Fukushima}, we know
$$
V_1=(1-L)^{-\frac{1}{2}},~~\text{so~that}~~F_{1,2}=D\big((1-L)^{\frac{1}{2}}\big)~~\text{and}~~\|f\|_{F_{1,2}}=|(1-L)^{\frac{1}{2}}f|_2.
$$
The dual space of $F_{1,2}$ is denoted by $F^*_{1,2}$. Denote the duality between $F^*_{1,2}$ and $F_{1,2}$ by
$_{F^*_{1,2}}\langle \cdot,\cdot\rangle_{F_{1,2}}$. $F^*_{1,2}$ is equipped with the following norm
\begin{eqnarray*}
 \|\eta\|_{F^*_{1,2}}:=\sup_{\substack{v\in
F_{1,2}\\ \|v\|_{F_{1,2}}\leq1}}\eta(v),~\eta\in F_{1,2}^*.
\end{eqnarray*}

\vspace{2mm}
The Dirichlet form $(\mathscr{E}, D(\mathscr{E}))$ of $(P_t)_{t>0}$ on
$L^2(\mu)$ is given by
\begin{eqnarray*}
   && D(\mathscr{E}):=D(\sqrt{-L}), \\
   && \mathscr{E}(u,v):=\langle\sqrt{-L}u, \sqrt{-L}v\rangle_2,
\end{eqnarray*}
and accordingly, we have the identification
$$F_{1,2}=D(\mathscr{E}),~~~\|u\|^2_{F_{1,2}}=\mathscr{E}_1(u,u),$$
where
$\mathscr{E}_\varepsilon:=\mathscr{E}+\varepsilon \langle\cdot, \cdot \rangle_2,~ \varepsilon\in(0,\infty)$, i.e.,
\begin{eqnarray*}
  \mathscr{E}_\varepsilon(v,v)=\|v\|^2_{F_{1,2,\varepsilon}}:=\mathscr{E}(v,v)+\varepsilon|v|^2_2,\ \text{for}\ v\in F_{1,2},
\end{eqnarray*}
and
\begin{eqnarray}\label{equivalent1}
 \|\eta\|_{F^*_{1,2,\varepsilon}}:= _{F^*_{1,2}}\langle \eta, (\varepsilon-L)^{-1}\eta\rangle^{\frac{1}{2}}_{F_{1,2}}:=\sup_{\substack{v\in
F_{1,2}\\ \|v\|_{F_{1,2,\varepsilon}}\leq1}}\eta(v),~\eta\in F_{1,2}^*,
\end{eqnarray}
which is equivalent to $\|\cdot\|_{F^*_{1,2}}$.


Let $H$ be a separable Hilbert space with inner product $\langle\cdot,\cdot\rangle_H$, and $H^*$ be its dual.
Let $V$ be a reflexive Banach space such that $V\subset H$ continuously and densely.
Then, for its dual space $V^*$, it follows that $H^*\subset V^*$ continuously and densely.
Identifying $H$ and $H^*$ via the Riesz isomorphism, we have that
$$V\subset H\subset V^*$$ continuously and densely.
If $_{V^*}\langle\cdot,\cdot\rangle_V$ denotes the dualization between $V^*$ and $V$, it follows that
\begin{eqnarray}\label{triple1}
_{V^*}\langle z,v\rangle_V=\langle z,v\rangle_H,~~\text{for~all}~z\in H,~~v\in V.
\end{eqnarray}
$(V,H,V^*)$ is called a Gelfand triple.

The first author of the current paper and her collaborators constructed a Gelfand triple with $V=L^2(\mu)$ and $H=F^*_{1,2}$ in \cite{RWX}. The Riesz map which identifies $F_{1,2}$ and $F^*_{1,2}$ is $(1-L)^{-1}: F^*_{1,2}\rightarrow F_{1,2}$.

We need the following lemma, which was proved in \cite[Lemma 2.2]{RWX}.
\begin{lemma}
The map
$$1-L:F_{1,2}\rightarrow F_{1,2}^*$$
extends to a linear isometry
$$1-L:L^2(\mu)\rightarrow(L^2(\mu))^*,$$
and for all $u,v\in L^2(\mu)$,
\begin{eqnarray}\label{triple2}
_{(L^2(\mu))^*}\langle(1-L)u, v\rangle_{L^2(\mu)}=\langle u, v \rangle_2.
\end{eqnarray}
\end{lemma}

\section{Hypotheses and main result}\label{Section3}
\setcounter{equation}{0}
 \setcounter{definition}{0}

In this section, we will present the hypotheses and the main result in this paper. Before giving the hypothesis, we need to introduce more basic notation.
For $\varpi\in(0,\infty)$,
define
\begin{eqnarray*}
\mathcal{H}^\varpi=&&\!\!\!\!\!\!\!\!\big\{h:[0,T]\times Z\rightarrow\Bbb{R}^+: \forall~ \Gamma\in\mathcal{B}([0,T])\otimes\mathcal{B}(Z)\nonumber\\
&&\ \ \ \ \ \ \ \ \ \ \text{with}~\nu_T(\Gamma)<\infty,~\text{we~have}~\int_\Gamma \exp(\varpi h(s,z))\nu(dz)ds<\infty\big\},
\end{eqnarray*}
and denote by $\mathcal{H}^\infty=\bigcap_{\varpi\in(0,\infty)}\mathcal{H}^\varpi$. Set
\begin{eqnarray*}
\mathcal{H}=&&\!\!\!\!\!\!\!\!\big\{h:[0,T]\times Z\rightarrow\Bbb{R}^+:\exists~ \delta>0,~ s.t.~ \forall~ \Gamma\in\mathcal{B}([0,T])\otimes\mathcal{B}(Z)\nonumber\\
&&\ \ \ \ \ \ \ \ \ \text{with}~\nu_T(\Gamma)<\infty,~\text{we~have}~\int_\Gamma \exp(\delta h^2(s,z))\nu(dz)ds<\infty\big\},
\end{eqnarray*}
then by \cite[Remark 3.2]{BCD}, we have
\begin{eqnarray}\label{eq H infinit H}
\mathcal{H}\subset\mathcal{H}^\infty.
\end{eqnarray}

Denote
\begin{eqnarray*}
L_2(\nu_T)=\big\{h:[0,T]\times Z\rightarrow\Bbb{R}^+:\int_0^{T}\int_Zh^2(s,z)\nu(dz)ds<\infty\big\}.
\end{eqnarray*}

In this paper, we study \eqref{eq:1} with the following hypotheses.

\vspace{2mm}
\noindent \textbf{(H1)} $\Psi(\cdot):\Bbb{R}\rightarrow \Bbb{R}$ is a monotonically nondecreasing Lipschitz function with $\Psi(0)=0$.

\vspace{2mm}
\noindent \textbf{(H2)} There exist $l_1\in\mathcal{H}^\infty\cap L_2(\nu_T)$, $l_2\in\mathcal{H}\cap L_2(\nu_T)$ and $l_3\in\mathcal{H}^{\varpi_0}\cap L_2(\nu_T)$ for some $\varpi_0>0$, such that
\begin{eqnarray*}
&&\textbf{(i)}~~\|f(t,x,z)-f(t,y,z)\|_{F^*_{1,2}}\leq l_1(t,z)\|x-y\|_{F^*_{1,2}},~\forall~(t,z)\in [0,T]\times Z,~x,y\in F^*_{1,2}.\\
&&\textbf{(ii)}~~\|f(t,x,z)\|_{F^*_{1,2}}\leq l_2(t,z)\big(\|x\|_{F^*_{1,2}}+1\big),~\forall~(t,x,z)\in [0,T]\times F^*_{1,2}\times Z.\\
&&\textbf{(iii)}~~|f(t,x,z)|_2\leq l_3(t,z)\big(|x|_2+1\big),~\forall~(t,x,z)\in [0,T]\times L^2(\mu)\times Z.\\
\end{eqnarray*}

\begin{remark}\label{equivalent}
From \cite[page:6, (2.6)]{RWX1} and \textbf{(H2)(i)}, we can easily see that for all $(t,z)\in [0,T]\times Z$, $x,y\in F^*_{1,2}$,
\begin{eqnarray}\label{f}
\|f(t,x,z)-f(t,y,z)\|_{F^*_{1,2,\varepsilon}}\leq \frac{l_1(t,z)}{\sqrt{\varepsilon}}\|x-y\|_{F^*_{1,2,\varepsilon}},~ \text{for~all}~0<\varepsilon<1.
\end{eqnarray}
\end{remark}

\begin{remark}\label{weak sense of H2}
\textbf{(H2)} implies the following strictly weak hypothesis:\\
\noindent \textbf{(H2)'} There exist functions $C_1, C_2, C_3\in L^1([0,T];\mathbb{R}^+)$ such that for all $t\in[0,T]$,
\begin{eqnarray*}
&&\textbf{(i)}~~\int_Z\|f(t,x,z)-f(t,y,z)\|^2_{F^*_{1,2}}\nu(dz)\leq C_1(t)\|x-y\|^2_{F^*_{1,2}},~\forall~x,y\in F^*_{1,2}.\\
&&\textbf{(ii)}~~\int_Z\|f(t,x,z)\|^2_{F^*_{1,2}}\nu(dz)\leq C_2(t)\big(\|x\|^2_{F^*_{1,2}}+1\big),~\forall~x\in F^*_{1,2}.\\
&&\textbf{(iii)}~~\int_Z|f(t,x,z)|^2_2\nu(dz)\leq C_3(t)\big(|x|^2_2+1\big),~\forall~x\in L^2(\mu).\\
\end{eqnarray*}
\end{remark}

From \cite[Theorem 3.1]{WZ21}, we can obtain the following theorem.

\begin{theorem}\label{th}
Suppose that \textbf{(H1)} and \textbf{(H2)'} hold. Then, for each $x\in L^2(\mu)$, there exists a unique solution $X^\epsilon$ to \eqref{eq:1}, i.e., $X^\epsilon=(X^\epsilon(t),t\in[0,T])$ is an $F^*_{1,2}$-valued c\`{a}dl\`{a}g $\mathcal{F}_t$-adapted process, and the following conditions are satisfied:
\begin{eqnarray}\label{eqn1}
{X^\epsilon}\in L^2([0,T]\times \Omega; L^2(\mu))\cap L^2(\Omega;L^\infty([0,T];F^*_{1,2}));
\end{eqnarray}
\begin{eqnarray}\label{eqn2}
\int_0^\cdot \Psi({X^\epsilon}(s))ds\in C([0,T];F_{1,2}),\ \Bbb{P}\text{-a.s}.;
\end{eqnarray}
and for all $t\in[0,T]$,
\begin{eqnarray}\label{eqn3}
X^\epsilon(t)=x+L\int_0^t\Psi({X^\epsilon}(s))ds+\epsilon\int_0^t\int_{Z}f(s,{X^\epsilon}(s-),z)\widetilde{N}^{\epsilon^{-1}}(dz,ds),
\end{eqnarray}
holds in $F^*_{1,2}$, $\Bbb{P}$-a.s..
\end{theorem}

The purpose of this paper is to establish a large deviation principle for \eqref{eq:1}, i.e., $X^\epsilon$ on $D([0,T];F^*_{1,2})$ as $\epsilon\rightarrow0$. Before the statement of the main result, we introduce the definition of LDP.
Let $\{\Gamma^\epsilon\}_{\epsilon>0}$ be a family of random variables defined on a probability space $(\Omega,\mathcal{F},\Bbb{P})$ and taking values in a Polish space $\mathcal{U}$.  The theory of large deviations is concerned with events $A\in\mathcal{B}(\mathcal{U})$ for which probability $\Bbb{P}(\Gamma^\epsilon\in A)$ converges to zero exponentially fast as $\epsilon\rightarrow0$. The exponential decay rate of such probabilities is typically expressed in terms of a ``rate function" as defined below.

\begin{definition}\label{definition2}
A function $I:\mathcal{U}\mapsto [0,\infty]$ is called a rate function on $\mathcal{U}$, if the level set $\{e\in\mathcal{U}:I(e)\leq M\}$ is a compact subset of $\mathcal{U}$ for each $M<\infty$. For $A\in\mathcal{B}(\mathcal{U})$, we define $I(A)=\inf_{e\in A}I(e)$.
\end{definition}

\begin{definition}\label{definition3}
Let $I$ be a rate function on $\mathcal{U}$. The sequence $\{\Gamma^\epsilon\}_{\epsilon>0}$ is said to satisfy a LDP on $\mathcal{U}$ as $\epsilon\rightarrow0$ with rate function $I$ if for each closed subset $F\subset\mathcal{U}$,
\begin{eqnarray*}
\lim\sup_{\epsilon\rightarrow0}\epsilon\log\Bbb{P}(\Gamma^\epsilon\in F)\leq-I(F),
\end{eqnarray*}
and for each open subset $G\subset\mathcal{U}$,
\begin{eqnarray*}
\lim\inf_{\epsilon\rightarrow0}\epsilon\log\Bbb{P}(\Gamma^\epsilon\in G)\geq-I(G).
\end{eqnarray*}
\end{definition}

We also need to introduce the so-called skeleton equation, which is used to define the rate function in our main result.
To do this, more basic notation should be introduced. Define $\Phi:[0,\infty)\rightarrow[0,\infty)$ by
\begin{eqnarray}\label{l}
\Phi(r)=r\log r-r+1,~~r\in[0,\infty).
\end{eqnarray}
For any $\varphi\in {\mathcal{A}}$ the quantity
\begin{eqnarray}
Q(\varphi)=\int_{Z_T}\Phi(\varphi(t,z))\nu_T(dtdz)
\end{eqnarray}
is defined as a $[0,\infty]$-valued random variable.

Let $N\in\Bbb{N}$, and define
\begin{eqnarray}\label{sn}
S^N=\{g:Z_T\rightarrow[0,\infty): Q(g)\leq N\}.
\end{eqnarray}
A function $g\in S^N$ can be identified with a measure $\nu_T^g\in\mathcal{M}_{FC}(Z_T)$, defined by
\begin{eqnarray}\label{eq Zhai 5}
\nu_T^g(A)=\int_Ag(s,z)\nu_T(dsdz),~~A\in\mathcal{B}(Z_T).
\end{eqnarray}
This identification induces a topology on $S^N$ under which $S^N$ is a compact space (\cite[Appendix]{BCD}). Throughout the paper, we use this topology on $S^N$. Denote $S=\bigcup_{N=1}^\infty S^N$.

For any $g\in S$, consider the following skeleton equation
\begin{equation}\label{eq:2}
dX^g(t)-L\Psi(X^g(t))dt=\int_Zf(t,X^g(t),z)\big(g(t,z)-1\big)\nu(dz)dt,~~\forall~t\in[0,T],
\end{equation}
with initial value $x\in F^*_{1,2}$. The existence and uniqueness of the solutions to \eqref{eq:2} will be presented in Section \ref{Section4}.

\vspace{2mm}

Our main theorem is as follows.
\begin{theorem}\label{Th1}
Fix $x\in L^2(\mu)$, and suppose that \textbf{(H1)} and \textbf{(H2)} hold. Then the solution of \eqref{eq:1}, i.e., $X^\epsilon$ satisfies a LDP on $D([0,T];F^*_{1,2})$ with the rate function $I: D([0,T];F^*_{1,2})\rightarrow[0,\infty]$ defined by
\begin{eqnarray*}
I(\phi)=\inf\Big\{Q(g):\phi=X^g, g\in S\Big\},\ \ \phi\in D([0,T];F^*_{1,2}),
\end{eqnarray*}
where $X^g$ is the solution to the skeleton equation \eqref{eq:2}.
Here, we use the convention that the infimum of an empty set is $\infty$.
\end{theorem}

The proof of Theorem \ref{Th1} will be given in Section \ref{Section5}. Before this, the proof of the existence and uniqueness of solutions to the skeleton equation \eqref{eq:2} will be given in Section \ref{Section4}. The following results will be used later, their proofs can be found in \cite[Lemmas 3.4 and 3.11]{BCD} (see the proofs of \cite[(3.3), (3.4), (3.5), (3.23), (3.26)]{BCD} respectively and use (\ref{eq H infinit H}) and the fact that for any compact subset $K\subset Z$, $\nu(K)<+\infty$).
\begin{lemma}\label{lemmal1}
\begin{enumerate}
       \item[(i)] Let $\chi\in\mathcal{H}\cap L_2(\nu_T)$, we have
    \begin{eqnarray}\label{eq 5}
    \sup_{g\in S^N}\int_0^T\int_Z\chi^2(s,z)(g(s,z)+1)\nu(dz)ds<\infty,~~~ \forall~ N\in\Bbb{N}.
    \end{eqnarray}
  \item[(ii)] Let $\chi\in\mathcal{H}^{\varpi}\cap L_2(\nu_T)$ for some $\varpi>0$, we have
    \begin{eqnarray}\label{eq 1}
    \sup_{g\in S^N}\int_0^T\int_Z\chi(s,z)|g(s,z)-1|\nu(dz)ds<\infty, ~~~ \forall~ N\in\Bbb{N}.
    \end{eqnarray}
    \item[(iii)]Let $\chi\in\mathcal{H}^{\infty}\cap L_2(\nu_T)$, $0\leq s\leq l\leq T$ and $N\in\Bbb{N}$, we have
  \begin{eqnarray}\label{eq 3}
    \lim_{\delta\rightarrow0}\sup_{g\in S^N}\sup_{|l-s|\leq \delta}\int_s^l\int_Z\chi(t,z)|g(t,z)-1|\nu(dz)dt=0.
    \end{eqnarray}
   \item[(iv)]Let $\chi\in\mathcal{H}^{\infty}\cap L_2(\nu_T)$ and $N\in\Bbb{N}$, then for any $\varepsilon>0$ there exists a compact set $K_\varepsilon\subset Z$ such that
\begin{eqnarray}\label{eq 4}
\sup_{g\in S^N}\int_0^T\int_{K_\varepsilon^c}\chi(s,z)|g(s,z)-1|\nu(dz)ds\leq\varepsilon,
\end{eqnarray}
where $K_\varepsilon^c$ is the complement of $K_\varepsilon$.
 \item[(v)] Let $\chi\in\mathcal{H}^\infty$, $K$ be a compact subset of $Z$ and $N\in\Bbb{N}$, we have
  \begin{eqnarray}\label{eq 2}
    \lim_{J\rightarrow\infty}\sup_{g\in S^N}\int_0^T\int_K \chi(s,z) 1_{\{\chi(s,z)\geq J\}}(s,z)g(s,z)\nu(dz)ds=0.
    \end{eqnarray}
\end{enumerate}
\end{lemma}

\section{Well-posedness for the skeleton equation}\label{Section4}
\setcounter{equation}{0}
 \setcounter{definition}{0}

\begin{definition}\label{definition1}
Given $g\in S$, a function $X^g$ is called a solution to \eqref{eq:2} if the following conditions are satisfied:
\begin{eqnarray}\label{defi1}
X^g\in C([0,T];F^*_{1,2})\cap L^2([0,T]; L^2(\mu));
\end{eqnarray}
\begin{eqnarray}\label{defi2}
\int_0^\cdot\Psi(X^g(s))ds\in C([0,T];F_{1,2});
\end{eqnarray}
and for all $t\in[0,T]$,
\begin{eqnarray}\label{defi3}
X^g(t)=x+L\int_0^t\Psi(X^g(s))ds+\int_0^t\int_Zf(s,X^g(s),z)\big(g(s,z)-1\big)\nu(dz)ds,
\end{eqnarray}
holds in $F^*_{1,2}$.
\end{definition}

The main result of this section is as follows.
\begin{theorem}\label{Th2}
Suppose that \textbf{(H1)} and \textbf{(H2)} hold. Then, for each $x\in L^2(\mu)$ and $g\in S^N$, there is a unique solution to \eqref{eq:2} in the sense of Definition \ref{definition1},
and there exists a constant $C_{N,T}\in(0,\infty)$, which depends on $N$ and $T$, such that
\begin{eqnarray}\label{defi4}
\sup_{g\in S^N}\sup_{t\in[0,T]}|X^g(t)|_2^2\leq C_{N,T}(|x|_2^2+1).
\end{eqnarray}
If \textbf{(H1)}, \textbf{(H2)(i)}, \textbf{(H2)(ii)} and the following are satisfied,
\begin{eqnarray}\label{psi}
\Psi(r)r\geq cr^2,
\end{eqnarray}
where $c\in(0,\infty)$. Then, for all $x\in F^*_{1,2}$, there is a unique solution $X^g$ to \eqref{eq:2} satisfying \eqref{defi1}-\eqref{defi3}.
\end{theorem}

The proof of Theorem \ref{Th2} is a generalization of \cite[Theorem 3.1]{RWX} and \cite[Theorem 3.1]{WZ21}. The main difference is that there is no diffusion term, but one more drift term, in \eqref{eq:2}, so appropriate estimates on the drift term are needed. Since the drift term is involved and we apply multistep approximation arguments, the proof of Theorem \ref{Th2} is painfully long.

To prove Theorem \ref{Th2}, we consider the following approximating equations for \eqref{eq:2}:
\begin{equation} \label{eq:3}
dX^{g}_{\varepsilon}(t)+(\varepsilon-L)\Psi(X^{g}_{\varepsilon}(t))dt=\int_{Z}f(t,X^{g}_{\varepsilon}(t),z)(g(t,z)-1)\nu(dz)dt,\ \text{in}\ [0,T],
\end{equation}
with initial value $X^{g}_{\varepsilon}(0)=x\in L^2(\mu)$, where $\varepsilon\in(0,1)$. We have the following proposition for \eqref{eq:3}.

\begin{proposition}\label{prop2}
Suppose that \textbf{(H1)} and \textbf{(H2)} hold. Then, for each $x\in L^2(\mu)$ and $g\in S^N$, there is a unique solution to \eqref{eq:3}, denoted by $X^{g}_{\varepsilon}$.
This solution has the following properties:
\begin{eqnarray}\label{p1}
X^{g}_{\varepsilon}\in L^2([0,T]; L^2(\mu))\cap C([0,T];F^*_{1,2});
\end{eqnarray}
\begin{eqnarray}\label{p2}
X^{g}_{\varepsilon}(t)+\!(\varepsilon-L)\!\int_0^t\!\Psi(X^{g}_{\varepsilon}(s))ds\!=\!x\!+\!\int_0^t\!\int_{Z}\!f(s,X^{g}_{\varepsilon}(s),z)\big(g(s,z)-1\big)\nu(dz)ds,\ \forall t\in[0,T]
\end{eqnarray}
holds in $F^*_{1,2}$.
Furthermore, there exists a constant $C_{N,T}\in(0,\infty)$, which is dependent on $N$ and $T$, but independent of $\varepsilon$, such that for all $N\in\Bbb{N}$,
\begin{eqnarray}\label{p3}
\sup_{g\in S^N}\sup_{t\in[0,T]}|X^{g}_{\varepsilon}(t)|_2^2\leq C_{N,T}\big(|x|_2^2+1\big).
\end{eqnarray}
If \textbf{(H1)}, \textbf{(H2)(i)}, \textbf{(H2)(ii)} and \eqref{psi} are satisfied. Then, for all $x\in F^*_{1,2}$, there is a unique solution to \eqref{eq:3} satisfying \eqref{p1} and \eqref{p2}.
\end{proposition}

\begin{proof}
First, let us consider the case with initial data $x\in F^*_{1,2}$ and assume that \textbf{(H1)}, \textbf{(H2)(i)}, \textbf{(H2)(ii)} and \eqref{psi} are satisfied.  Set $V:=L^2(\mu)$ and $H:=F^*_{1,2}$. For $u\in V$, define
\begin{eqnarray*}
A(t,u):=(L-\varepsilon)\Psi(u)+\int_Zf(t,u,z)\big(g(t,z)-1\big)\nu(dz).
\end{eqnarray*}
Notice that from \textbf{(H2)(ii)} we have
\begin{eqnarray*}
     &&\!\!\!\!\!\!\!\!\int_0^T\int_Z\|f(t,u,z)\big(g(t,z)-1\big)\|_{F^*_{1,2}}\nu(dz)dt\nonumber\\
\leq&&\!\!\!\!\!\!\!\!\big(\|u\|_{F^*_{1,2}}+1\big)\int_0^T\int_Z l_2(t,z)|g(t,z)-1|\nu(dz)dt,
\end{eqnarray*}
and from \eqref{eq 1} and (\ref{eq H infinit H})
we know that, there exists a constant $C_{N}$ such that
\begin{eqnarray}\label{h0}
C_{N}:=\sup_{i=1,2,3}\sup_{g\in S^N}\int_{Z_T}l_i(s,z)|g(s,z)-1|\nu(dz)ds<\infty.
\end{eqnarray}
Then, $A(t,u)$ can be regarded as a mapping from $[0,T]\times L^2(\mu)$ to $(L^2(\mu))^*$. Under the Gelfand triple $L^2(\mu)\subset F^*_{1,2}\equiv F_{1,2}\subset (L^2(\mu))^*$, we can check the four conditions in \cite[Theorem 1.2]{BLZ} to get the existence and uniqueness of solutions to \eqref{eq:3}. As mentioned before, compared with \cite[Lemma 3.1, Step 1]{RWX}, our $A$ here has one more drift term $\int_{Z}f(t,u,z)\big(g(t,z)-1\big)\nu(dz)$, so we only need to estimate that term.

\vspace{2mm}
\textbf{(i) }\textbf{Hemicontinuity}
\vspace{2mm}

Let $u,v,w\in V(:=L^2(\mu))$. We need to show that for $\iota\in \Bbb{R},~ |\iota|\leq1$,
$$
\lim_{\iota\rightarrow0}~_{(L^2(\mu))^*}\big\langle A(t,u+\iota v),w\big\rangle_{L^2(\mu)}-~ _{(L^2(\mu))^*}\big\langle A(t,u),w\big\rangle_{L^2(\mu)}=0.
$$
Since in \cite[ Lemma 3.1, Step 1]{RWX}, the authors proved that
$$\lim_{\iota\rightarrow0}~ _{(L^2(\mu))^*}\big\langle (L-\varepsilon)\Psi(u+\iota v),w\big\rangle_{L^2(\mu)}-~ _{(L^2(\mu))^*}\big\langle (L-\varepsilon)\Psi(u),w\big\rangle_{L^2(\mu)}=0,$$
here we only need to prove
\begin{eqnarray}\label{hemi}
&&\!\!\!\!\!\!\!\!\lim_{\iota\rightarrow0}~ _{(L^2(\mu))^*}\big\langle \int_Z f(t,u+\iota v,z)\big(g(t,z)-1\big)\nu(dz),w\big\rangle_{L^2(\mu)}\nonumber\\
&&\!\!\!\!\!\!\!\!-_{(L^2(\mu))^*}\big\langle \int_Z f(t,u,z)(g(t,z)-1)\nu(dz),w\big\rangle_{L^2(\mu)}=0.
\end{eqnarray}
By \eqref{triple1} and \textbf{(H2)(i)}, we have
\begin{eqnarray}\label{hemi1}
&&\!\!\!\!\!\!\!\! _{(L^2(\mu))^*}\big\langle \int_Z \big(f(t,u+\iota v,z)-f(t,u,z)\big)\big(g(t,z)-1\big)\nu(dz), w\big\rangle_{L^2(\mu)}\nonumber\\
=&&\!\!\!\!\!\!\!\!\big\langle\int_Z \big(f(t,u+\iota v,z)-f(t,u,z)\big)\big(g(t,z)-1\big)\nu(dz), w\big\rangle_{F^*_{1,2}}\nonumber\\
\leq&&\!\!\!\!\!\!\!\!\big\|\int_Z \big(f(t,u+\iota v,z)-f(t,u,z)\big)\big(g(t,z)-1\big)\nu(dz)\big\|_{F^*_{1,2}} \cdot \|w\|_{F^*_{1,2}}\nonumber\\
\leq&&\!\!\!\!\!\!\!\!\iota\cdot\|v\|_{F^*_{1,2}}\cdot\|w\|_{F^*_{1,2}}\cdot\int_Z l_1(t,z)|g(t,z)-1|\nu(dz).
\end{eqnarray}
Denoting
\begin{eqnarray}\label{h1}
h_1(t):=\int_Z l_1(t,z)|g(t,z)-1|\nu(dz),
\end{eqnarray}
from \eqref{h0}, we know that $h_1\in L^1([0,T];\Bbb{R}^+)$ and
\begin{eqnarray}\label{hemi2}
\iota\cdot\|v\|_{F^*_{1,2}}\cdot\|w\|_{F^*_{1,2}}\cdot h_1(t)\rightarrow0,~\text{as}~\iota\rightarrow0,~\lambda_T\text{-a.s. on }[0,T],
\end{eqnarray}
which implies \eqref{hemi}.

\vspace{2mm}
\textbf{(ii)} \textbf{Local Monotonicity}
\vspace{2mm}

Let $u,v\in V(:=L^2(\mu))$.
By \eqref{triple1}, \textbf{(H2)(i)}, \eqref{h1}, and \eqref{h0}, we have
\begin{eqnarray*}
&&\!\!\!\!\!\!\!\! _{(L^2(\mu))^*}\big\langle \int_Z \big(f(t,u,z)-f(t,v,z)\big)\big(g(t,z)-1\big)\nu(dz), u-v\big\rangle_{L^2(\mu)}\nonumber\\
=&&\!\!\!\!\!\!\!\!\big\langle\int_Z \big(f(t,u,z)-f(t,v,z)\big)\big(g(t,z)-1\big)\nu(dz), u-v\big\rangle_{F^*_{1,2}}\nonumber\\
\leq&&\!\!\!\!\!\!\!\!\big\|\int_Z \big(f(t,u,z)-f(t,v,z)\big)\big(g(t,z)-1\big)\nu(dz)\big\|_{F^*_{1,2}} \cdot \|u-v\|_{F^*_{1,2}}\nonumber\\
\leq&&\!\!\!\!\!\!\!\!h_1(t)\cdot\|u-v\|^2_{F^*_{1,2}}.
\end{eqnarray*}
Let $Lip\Psi$ denote the Lipschitz constant of $\Psi$.
From \cite[page: 2138, (ii)]{RWX}, we know that
\begin{eqnarray*}
&&\!\!\!\!\!\!\!\!_{(L^2(\mu))^*}\big\langle(L-\varepsilon)(\Psi(u)-\Psi(v)), u-v\big\rangle_{L^2(\mu)}\nonumber\\
\leq&&\!\!\!\!\!\!\!\!\frac{(1-\varepsilon)^2}{\widetilde{\alpha}}\cdot \|u-v\|^2_{F^*_{1,2}},
\end{eqnarray*}
where
\begin{eqnarray*}
\widetilde{\alpha}:=(k+1)^{-1},~~~~k:=Lip\Psi.
\end{eqnarray*}
So
\begin{eqnarray}\label{monoto1}
&&_{(L^2(\mu))^*}\big\langle A(t,u)-A(t,v),u-v\big\rangle_{L^2(\mu)}\leq\Big(\frac{(1-\varepsilon)^2}{\widetilde{\alpha}}+h_1(t)\Big)\cdot \|u-v\|^2_{F^*_{1,2}},
\end{eqnarray}
which implies the local monotonicity.

\vspace{2mm}
\textbf{(iii)} \textbf{Coercivity}
\vspace{2mm}

Let $u\in V(:=L^2(\mu))$. By \eqref{triple1} and \textbf{(H2)(ii)},
\begin{eqnarray*}
&&\!\!\!\!\!\!\!\! _{(L^2(\mu))^*}\big\langle \int_Z f(t,u,z)\big(g(t,z)-1\big)\nu(dz), u\big\rangle_{L^2(\mu)}\nonumber\\
=&&\!\!\!\!\!\!\!\!\big\langle\int_Z f(t,u,z)\big(g(t,z)-1\big)\nu(dz), u\big\rangle_{F^*_{1,2}}\nonumber\\
\leq&&\!\!\!\!\!\!\!\!\big\|\int_Z f(t,u,z)\big(g(t,z)-1\big)\nu(dz)\big\|_{F^*_{1,2}} \cdot \|u\|_{F^*_{1,2}}\nonumber\\
\leq&&\!\!\!\!\!\!\!\!\int_Zl_2(t,z)|g(t,z)-1|\nu(dz)\cdot \big(\|u\|_{F^*_{1,2}}+1\big)\|u\|_{F^*_{1,2}},
\end{eqnarray*}
Denoting
\begin{eqnarray*}
h_2(t):=\int_Z l_2(t,z)\big|g(t,z)-1\big|\nu(dz),
\end{eqnarray*}
from \eqref{h0}, we know that $h_2\in L^1([0,T];\Bbb{R}^+)$ and
\begin{eqnarray*}
&&\!\!\!\!\!\!\!\!\int_Zl_2(t,z)\big|g(t,z)-1\big|\nu(dz)\cdot \big(\|u\|_{F^*_{1,2}}+1\big)\|u\|_{F^*_{1,2}}\\
\leq&&\!\!\!\!\!\!\!\! 2h_2(t)\cdot\big(\|u\|^2_{F^*_{1,2}}+1\big).
\end{eqnarray*}
From \cite[page:2138-2139, (iii)]{RWX},  for any $\sigma>0$,
\begin{eqnarray*}
&&\!\!\!\!\!\!\!\!_{(L^2(\mu))^*}\big\langle (\varepsilon-L)\Psi(u),u\big\rangle_{L^2(\mu)}\nonumber\\
\leq&&\!\!\!\!\!\!\!\!\Big(-c+\sigma^2k^2(1-\varepsilon)\Big)\cdot|u|_2^2+\frac{(1-\varepsilon)}{\sigma^2}\cdot\|u\|^2_{F^*_{1,2}}.
\end{eqnarray*}
Therefore,
\begin{eqnarray}\label{coercivity1}
&&\!\!\!\!\!\!\!\!_{(L^2(\mu))^*}\big\langle A(t,u),u\big\rangle_{L^2(\mu)}\nonumber\\
\leq&&\!\!\!\!\!\!\!\!\Big(-c+\sigma^2k^2(1-\varepsilon)\Big)\cdot|u|_2^2+\Big(\frac{(1-\varepsilon)}{\sigma^2}+2h_2(t)\Big)\cdot\big(\|u\|^2_{F^*_{1,2}}+1\big).
\end{eqnarray}
Choosing $\sigma$ small enough, $-c+\sigma^2k^2(1-\varepsilon)$ becomes
negative, which implies the coercivity.

\vspace{2mm}
\textbf{(iv) Growth}

Let $u\in V(:=L^2(\mu))$. Notice that
\begin{eqnarray*}
\|A(t,u)\|_{(L^2(\mu))^*}=\sup_{|v|_2=1}~_{(L^2(\mu))^*}\big\langle(L-\varepsilon)(\Psi(u))+\int_Zf(t,u,z)\big(g(t,z)-1\big)\nu(dz)ds,v\big\rangle_{L^2(\mu)}.
\end{eqnarray*}
From \cite[page:2139, (iv)]{RWX}, we know that
$$\|(L-\varepsilon)\Psi(u)\|_{(L^2(\mu))^*}\leq2k|u|_2.$$
Since $L^2(\mu)\subset F^*_{1,2}\subset (L^2(\mu))^*$ continuously and densely, from \textbf{(H2)(ii)} we get
$$_{(L^2(\mu))^*}\big\langle \int_Z f(t,u,z)\big(g(t,z)-1\big)\nu(dz),v\big\rangle_{L^2(\mu)}\leq h_2(t)\cdot(|u|_2+1)\cdot|v|_2,$$
so
\begin{eqnarray}\label{growth1}
\|A(t,u)\|_{(L^2(\mu))^*}\leq\big(2k+h_2(t)\big)\cdot(|u|_2+1).
\end{eqnarray}
Hence the growth holds.

Then by \cite[Theorem 1.2]{BLZ}, there exists a unique solution to \eqref{eq:3}, which we denote as $X^{g}_{\varepsilon}$, which takes value in $F^*_{1,2}$ and satisfies \eqref{p1} and \eqref{p2}.
\end{proof}

\begin{remark}\label{remarkh}
As shown above,
the coefficient in the right-hand sides of \eqref{monoto1} has a term with $h_1(t)$,
and
\eqref{coercivity1} and \eqref{growth1} have terms with $h_2(t)$.
These look different from the conditions in \cite[Theorem 1.2]{BLZ}, where the coefficients are constants.
However,
since \eqref{h0} holds, we know that both $\int_{0}^{T}h_1(t)dt$ and $\int_{0}^{T}h_2(t)dt$ are bounded, so by using the same idea as in the proof of \cite[Theorem 1.2]{BLZ}, it is not difficult to get Proposition \ref{prop2}. Accordingly, here we still regard (i)-(iv) as the corresponding conditions of \cite[Theorem 1.2]{BLZ}.
\end{remark}

\vspace{2mm}
In the following, we consider the case that \eqref{psi} is not satisfied.
If $x\in F^*_{1,2}$, but \eqref{psi} is not satisfied,
then (i), (ii), and (iv) still hold, but (iii) may not be true in general. In this case, we shall approximate $\Psi$ by $\Psi+\delta \textbf{1}$, $\delta\in(0,1)$, i.e., we consider the following approximating equations:
\begin{eqnarray}\label{eq:4}
dX^{g}_{\varepsilon,\delta}(t)=&&\!\!\!\!\!\!\!\!(L-\varepsilon)\big(\Psi(X^{g}_{\varepsilon,\delta}(t))+\delta X^{g}_{\varepsilon,\delta}(t)\big)dt\nonumber\\
&&\!\!\!\!\!\!\!\!+\int_{Z}f(t,X^{g}_{\varepsilon,\delta}(t),z)\big(g(t,z)-1\big)\nu(dz)dt,\ \text{in}\ [0,T],
\end{eqnarray}
with initial value $X^{g}_{\varepsilon,\delta}(0)=x\in L^2(\mu)$. In this case, since $L^2(\mu)\subset F^*_{1,2}$ continuously and densely, the initial value can be regarded as in $F^*_{1,2}$, and obviously $(\Psi(r)+\delta r)r\geq \delta r^2$, which corresponds to the condition \eqref{psi}. Then, by similar arguments as above, it is easy to prove that if $x\in F^*_{1,2}$, there exists a unique solution $X^{g}_{\varepsilon,\delta}$ to \eqref{eq:4} such that $X^{g}_{\varepsilon,\delta}\in L^2\big([0,T];L^2(\mu)\big)\cap C([0,T];F_{1,2}^*)$, and for all $t\in[0,T]$,
\begin{eqnarray}\label{eq:4.1}
X^{g}_{\varepsilon,\delta}(t)=&&\!\!\!\!\!\!\!\!x+(L-\varepsilon)\int_0^t(\Psi(X^{g}_{\varepsilon,\delta}(s))+\delta X^{g}_{\varepsilon,\delta}(s))ds\nonumber\\
&&\!\!\!\!\!\!\!\!+\int_0^t\int_{Z}f(t,X^{g}_{\varepsilon,\delta}(s),z)\big(g(s,z)-1\big)\nu(dz)ds
\end{eqnarray}
holds in $F^*_{1,2}$, and
\begin{eqnarray}\label{eq:4.2}
\sup_{t\in[0,T]}\|X^{g}_{\varepsilon,\delta}(t)\|^2_{F^*_{1,2}}<\infty.
\end{eqnarray}
Next, we want to prove that the sequence $\{X^{g}_{\varepsilon,\delta}\}$
converges to the solution of \eqref{eq:3} as $\delta\to 0$. From now on,
we assume that the initial value $x\in L^2(\mu)$, and we have the following result for \eqref{eq:4}.

\begin{claim}\label{claim1}
For any $\varepsilon, \delta\in(0,1), g\in S^N, t\in[0,T],$
\begin{eqnarray*}
\sup_{s\in[0,t]}|X^{g}_{\varepsilon,\delta}(s)|_2^2+4\delta\int_0^t\|X^{g}_{\varepsilon,\delta}(s)\|^2_{F_{1,2}}ds\leq C_{N,T}(|x|_2^2+1).
\end{eqnarray*}
\end{claim}
\begin{proof}
Rewrite \eqref{eq:4} to the following:
\begin{eqnarray*}
X^{g}_{\varepsilon,\delta}(t)=&&\!\!\!\!\!\!\!\!x+\int_0^t(L-\varepsilon)\big(\Psi(X^{g}_{\varepsilon,\delta}(s))+\delta X^{g}_{\varepsilon,\delta}(s)\big)ds\nonumber\\
&&\!\!\!\!\!\!\!\!+\int_0^t\int_{Z}f(s,X^{g}_{\varepsilon,\delta}(s),z)\big(g(s,z)-1\big)\nu(dz)ds,\ \forall~t\in[0,T].
\end{eqnarray*}
For $\alpha>\varepsilon$, applying the operator $(\alpha-L)^{-\frac{1}{2}}:F^*_{1,2}\rightarrow L^2(\mu)$ to both sides of the above equation, we get
\begin{eqnarray*}
(\alpha-L)^{-\frac{1}{2}}X^{g}_{\varepsilon,\delta}(t)=&&\!\!\!\!\!\!\!\!(\alpha-L)^{-\frac{1}{2}}x+\int_0^t(L-\varepsilon)(\alpha-L)^{-\frac{1}{2}}\big(\Psi(X^{g}_{\varepsilon,\delta}(s))+\delta X^{g}_{\varepsilon,\delta}(s)\big)ds\nonumber\\
&&\!\!\!\!\!\!\!\!+\int_0^t\int_{Z}(\alpha-L)^{-\frac{1}{2}}f(s,X^{g}_{\varepsilon,\delta}(s),z)\big(g(s,z)-1\big)\nu(dz)ds.
\end{eqnarray*}
Applying the chain rule in $L^2(\mu)$, we obtain that for $t\in[0,T]$,
\begin{eqnarray}\label{c1}
&&\!\!\!\!\!\!\!\!\big|(\alpha-L)^{-\frac{1}{2}}X^{g}_{\varepsilon,\delta}(t)\big|_2^2\nonumber\\
=&&\!\!\!\!\!\!\!\!|(\alpha-L)^{-\frac{1}{2}}x|_2^2+2\int_0^t~_{F^*_{1,2}}\Big\langle (L-\varepsilon)(\alpha-L)^{-\frac{1}{2}}\big(\Psi(X^{g}_{\varepsilon,\delta}(s))\big),(\alpha-L)^{-\frac{1}{2}}X^{g}_{\varepsilon,\delta}(s)\Big\rangle_{F_{1,2}}ds\nonumber\\
&&\!\!\!\!\!\!\!\!+2\delta\int_0^t~_{F^*_{1,2}}\Big\langle(L-\varepsilon)(\alpha-L)^{-\frac{1}{2}}X^{g}_{\varepsilon,\delta}(s),(\alpha-L)^{-\frac{1}{2}}X^{g}_{\varepsilon,\delta}(s)\Big\rangle_{F_{1,2}}ds\\
&&\!\!\!\!\!\!\!\!+2\int_0^t~_{F^*_{1,2}}\Big\langle\int_{Z}(\alpha-L)^{-\frac{1}{2}}f(s,X^{g}_{\varepsilon,\delta}(s),z)\big(g(s,z)-1\big)\nu(dz),(\alpha-L)^{-\frac{1}{2}}X^{g}_{\varepsilon,\delta}(s)\Big\rangle_{F_{1,2}}ds.\nonumber
\end{eqnarray}
From \cite[(3.19)]{RWX} and the last two lines on page 2140 of \cite{RWX},
we know that
\begin{eqnarray}\label{c2}
&&\!\!\!\!\!\!\!\!2\int_0^t~_{F^*_{1,2}}\Big\langle (L-\varepsilon)(\alpha-L)^{-\frac{1}{2}}\big(\Psi(X^{g}_{\varepsilon,\delta}(s))\big),(\alpha-L)^{-\frac{1}{2}}X^{g}_{\varepsilon,\delta}(s)\Big\rangle_{F_{1,2}}ds\leq0.
\end{eqnarray}

For the third term in the right hand-side of \eqref{c1}, from \eqref{triple1} and since $\|u\|_{F_{1,2}}\geq|u|_2,\ \forall u\in F_{1,2}$,
\begin{eqnarray}\label{c3}
&&\!\!\!\!\!\!\!\!2\delta\int_0^t~_{F^*_{1,2}}\Big\langle(L-\varepsilon)(\alpha-L)^{-\frac{1}{2}}X^{g}_{\varepsilon,\delta}(s),(\alpha-L)^{-\frac{1}{2}}X^{g}_{\varepsilon,\delta}(s)\Big\rangle_{F_{1,2}}ds\nonumber\\
=&&\!\!\!\!\!\!\!\!-2\delta\int_0^t~_{F^*_{1,2}}\Big\langle(1-L)(\alpha-L)^{-\frac{1}{2}}X^{g}_{\varepsilon,\delta}(s),(\alpha-L)^{-\frac{1}{2}}X^{g}_{\varepsilon,\delta}(s)\Big\rangle_{F_{1,2}}ds\nonumber\\
&&\!\!\!\!\!\!\!\!+(1-\varepsilon)2\delta\int_0^t~_{F^*_{1,2}}\Big\langle(\alpha-L)^{-\frac{1}{2}}X^{g}_{\varepsilon,\delta}(s),(\alpha-L)^{-\frac{1}{2}}X^{g}_{\varepsilon,\delta}(s)\Big\rangle_{F_{1,2}}ds\nonumber\\
=&&\!\!\!\!\!\!\!\!-2\delta\int_0^t~\Big\langle(\alpha-L)^{-\frac{1}{2}}X^{g}_{\varepsilon,\delta}(s),(\alpha-L)^{-\frac{1}{2}}X^{g}_{\varepsilon,\delta}(s)\Big\rangle_{F_{1,2}}ds\nonumber\\
&&\!\!\!\!\!\!\!\!+(1-\varepsilon)2\delta\int_0^t\Big\langle(\alpha-L)^{-\frac{1}{2}}X^{g}_{\varepsilon,\delta}(s),(\alpha-L)^{-\frac{1}{2}}X^{g}_{\varepsilon,\delta}(s)\Big\rangle_2ds\nonumber\\
\leq&&\!\!\!\!\!\!\!\!-2\delta\int_0^t\|(\alpha-L)^{-\frac{1}{2}}X^{g}_{\varepsilon,\delta}(s)\|^2_{F_{1,2}}ds+2\int_0^t|(\alpha-L)^{-\frac{1}{2}}X^{g}_{\varepsilon,\delta}(s)|_2^2ds.
\end{eqnarray}

For the fourth term in the right hand-side of \eqref{c1}, notice that $\int_{Z}(\alpha-L)^{-\frac{1}{2}}f(s,X^{g}_{\varepsilon,\delta}(s),z)\\\big(g(s,z)-1\big)\nu(dz)$ takes a value in the space $L^2(\mu)$, so by \eqref{triple1} we have
\begin{eqnarray}\label{c4}
&&\!\!\!\!\!\!\!\!_{F^*_{1,2}}\Big\langle\int_Z(\alpha-L)^{-\frac{1}{2}}f(s,X^{g}_{\varepsilon,\delta}(s),z)\big(g(s,z)-1\big)\nu(dz),(\alpha-L)^{-\frac{1}{2}}X^{g}_{\varepsilon,\delta}(s)\Big\rangle_{F_{1,2}}\nonumber\\
=&&\!\!\!\!\!\!\!\!\Big\langle\int_Z(\alpha-L)^{-\frac{1}{2}}f(s,X^{g}_{\varepsilon,\delta}(s),z)\big(g(s,z)-1\big)\nu(dz),(\alpha-L)^{-\frac{1}{2}}X^{g}_{\varepsilon,\delta}(s)\Big\rangle_2.
\end{eqnarray}

Multiplying both sides of \eqref{c1} by $\alpha$, and taking the above estimates into account, by \textbf{(H2)(iii)}, since $\sqrt{\alpha}(\alpha-L)^{-\frac{1}{2}}$ is a contraction on $L^2(\mu)$, \eqref{c1} yields that for all $t\in[0,T]$,
\begin{eqnarray*}\label{c5}
&&\!\!\!\!\!\!\!\!\Big|\sqrt{\alpha}(\alpha-L)^{-\frac{1}{2}}X^{g}_{\varepsilon,\delta}(t)\Big|_2^2+2\delta\int_0^t\|\sqrt{\alpha}(\alpha-L)^{-\frac{1}{2}}X^{g}_{\varepsilon,\delta}(s)\|^2_{F_{1,2}}ds\nonumber\\
\leq&&\!\!\!\!\!\!\!\!|\sqrt{\alpha}(\alpha-L)^{-\frac{1}{2}}x|_2^2\nonumber\\
&&\!\!\!\!\!\!\!\!+2\int_0^t\Big|\int_Z\sqrt{\alpha}(\alpha-L)^{-\frac{1}{2}}f(s,X^{g}_{\varepsilon,\delta}(s),z)\big(g(s,z)-1\big)\nu(dz)\Big|_2\cdot|\sqrt{\alpha}(\alpha-L)^{-\frac{1}{2}}X^{g}_{\varepsilon,\delta}(s)|_2ds\nonumber\\
&&\!\!\!\!\!\!\!\!+2\int_0^t|\sqrt{\alpha}(\alpha-L)^{-\frac{1}{2}}X^{g}_{\varepsilon,\delta}(s)|_2^2ds\nonumber\\
\leq&&\!\!\!\!\!\!\!\!|\sqrt{\alpha}(\alpha-L)^{-\frac{1}{2}}x|_2^2\nonumber\\
&&\!\!\!\!\!\!\!\!+2\int_0^t\!\!\big|\!\int_Z\!f(s,X^{g}_{\varepsilon,\delta}(s),z)\big(g(s,z)-1\big)\nu(dz)\big|_2\cdot|\sqrt{\alpha}(\alpha-L)^{-\frac{1}{2}}X^{g}_{\varepsilon,\delta}(s)|_2ds+2\int_0^t\!\!|X^{g}_{\varepsilon,\delta}(s)|_2^2ds\nonumber\\
\leq&&\!\!\!\!\!\!\!\!|\sqrt{\alpha}(\alpha-L)^{-\frac{1}{2}}x|_2^2\nonumber\\
&&\!\!\!\!\!\!\!\!+2\int_0^t\sqrt{h_3(s)}\big(|X^{g}_{\varepsilon,\delta}(s)|_2+1\big)\cdot \sqrt{h_3(s)}\cdot|\sqrt{\alpha}(\alpha-L)^{-\frac{1}{2}}X^{g}_{\varepsilon,\delta}(s)|_2ds+2\int_0^t|X^{g}_{\varepsilon,\delta}(s)|_2^2ds,
\end{eqnarray*}
here and in the sequel
$$h_3(s):=\int_Zl_3(s,z)\big|g(s,z)-1\big|\nu(dz).$$
From \eqref{h0}, we know that $\int_0^Th_3(s)ds\leq C_{l_3,N}$,
and by Young's inequality, we then get
\begin{eqnarray}\label{c6}
&&\!\!\!\!\!\!\!\!\sup_{s\in[0,t]}\Big|\sqrt{\alpha}(\alpha-L)^{-\frac{1}{2}}X^{g}_{\varepsilon,\delta}(s)\Big|_2^2+2\delta\int_0^t\|\sqrt{\alpha}(\alpha-L)^{-\frac{1}{2}}X^{g}_{\varepsilon,\delta}(s)\|^2_{F_{1,2}}ds\nonumber\\
\leq&&\!\!\!\!\!\!\!\!|\sqrt{\alpha}(\alpha-L)^{-\frac{1}{2}}x|_2^2+2C_{l_3,N}\int_0^th_3(s)\big(|X^{g}_{\varepsilon,\delta}(s)|_2+1\big)^2ds\nonumber\\
&&\!\!\!\!\!\!\!\!+\frac{1}{2C_{l_3,N}}\int_0^th_3(s)|\sqrt{\alpha}(\alpha-L)^{-\frac{1}{2}}X^{g}_{\varepsilon,\delta}(s)|^2_2ds+2\int_0^t|X^{g}_{\varepsilon,\delta}(s)|_2^2ds\nonumber\\
\leq&&\!\!\!\!\!\!\!\!|\sqrt{\alpha}(\alpha-L)^{-\frac{1}{2}}x|_2^2+4C_{l_3,N}\int_0^th_3(s)\big(|X^{g}_{\varepsilon,\delta}(s)|^2_2+1\big)ds+2\int_0^t|X^{g}_{\varepsilon,\delta}(s)|_2^2+1ds\nonumber\\
&&\!\!\!\!\!\!\!\!+\sup_{s\in[0,t]}\Big|\sqrt{\alpha}(\alpha-L)^{-\frac{1}{2}}X^{g}_{\varepsilon,\delta}(s)\Big|_2^2\cdot \frac{1}{2C_{l_3,N}}\int_0^th_3(s)ds\nonumber\\
\leq&&\!\!\!\!\!\!\!\!|\sqrt{\alpha}(\alpha-L)^{-\frac{1}{2}}x|_2^2+\int_0^t\big(4C_{l_3,N}h_3(s)+2\big)\big(|X^{g}_{\varepsilon,\delta}(s)|^2_2+1\big)ds\nonumber\\
&&\!\!\!\!\!\!\!\!+\frac{1}{2}\sup_{s\in[0,t]}\Big|\sqrt{\alpha}(\alpha-L)^{-\frac{1}{2}}X^{g}_{\varepsilon,\delta}(s)\Big|_2^2.
\end{eqnarray}

Since $|\sqrt{\alpha}(\alpha-L)^{-\frac{1}{2}}\cdot|_2$ is equivalent to $\|\cdot\|_{F^*_{1,2}}$,
the first term in the left hand-side of \eqref{c6} is finite by \eqref{eq:4.2}.
Then, \eqref{c6} yields that
\begin{eqnarray}\label{c7}
&&\!\!\!\!\!\!\!\!\sup_{s\in[0,t]}\Big|\sqrt{\alpha}(\alpha-L)^{-\frac{1}{2}}X^{g}_{\varepsilon,\delta}(s)\Big|_2^2+4\delta\int_0^t\|\sqrt{\alpha}(\alpha-L)^{-\frac{1}{2}}X^{g}_{\varepsilon,\delta}(s)\|^2_{F_{1,2}}ds\nonumber\\
\leq&&\!\!\!\!\!\!\!\!2|\sqrt{\alpha}(\alpha-L)^{-\frac{1}{2}}x|_2^2+\int_0^t(8C_{l_3,N}h_3(s)+4\big)\big(|X^{g}_{\varepsilon,\delta}(s)|^2_2+1\big)ds.
\end{eqnarray}

By exactly the same arguments as in \cite[Page:15, (3.32)]{RWXSPA}, we have
\begin{eqnarray*}
&&\!\!\!\!\!\!\!\!\sup_{s\in[0,t]}\big|X^{g}_{\varepsilon,\delta}(s)\big|_2^2+4\delta\int_0^t\|X^{g}_{\varepsilon,\delta}(s)\|^2_{F_{1,2}}ds\nonumber\\
\leq&&\!\!\!\!\!\!\!\!\liminf_{\alpha\rightarrow\infty}\Big[\sup_{s\in[0,t]}\Big|\sqrt{\alpha}(\alpha-L)^{-\frac{1}{2}}X^{g}_{\varepsilon,\delta}(s)\Big|_2^2+4\delta\int_0^t\|\sqrt{\alpha}(\alpha-L)^{-\frac{1}{2}}X^{g}_{\varepsilon,\delta}(s)\|^2_{F_{1,2}}ds\Big]\nonumber\\
\leq&&\!\!\!\!\!\!\!\!\liminf_{\alpha\rightarrow\infty}\Big[2|\sqrt{\alpha}(\alpha-L)^{-\frac{1}{2}}x|_2^2+\int_0^t(8C_{l_3,N}h_3(s)+4\big)\big(|X^{g}_{\varepsilon,\delta}(s)|^2_2+1\big)ds\Big]\nonumber\\
=&&\!\!\!\!\!\!\!\!2|x|_2^2+\int_0^t(8C_{l_3,N}h_3(s)+4\big)|X^{g}_{\varepsilon,\delta}(s)|^2_2ds+\int_0^t8C_{l_3,N}h_3(s)+4ds.
\end{eqnarray*}
Then by Gronwall's lemma, we know that for all $\varepsilon, \delta\in(0,1)$, $t\in[0,T]$,
\begin{eqnarray}
&&\!\!\!\!\!\!\!\!\sup_{s\in[0,t]}\big|X^{g}_{\varepsilon,\delta}(s)\big|_2^2+4\delta\int_0^t\|X^{g}_{\varepsilon,\delta}(s)\|^2_{F_{1,2}}ds\nonumber\\
\leq&&\!\!\!\!\!\!\!\!\Big(2|x|_2^2+8C_{l_3,N}\int_0^th_3(s)ds+4T\Big)\cdot e^{8C_{l_3,N}\int_0^th_3(s)ds+4T}\nonumber\\
\leq&&\!\!\!\!\!\!\!\!\Big(2|x|_2^2+8C_{l_3,N}^2+4T\Big)e^{8C_{l_3,N}^2+4T}\label{eq zhai 3}\\
:=&&\!\!\!\!\!\!\!\!C_{N,T}(|x|_2^2+1).\nonumber
\end{eqnarray}
\hspace{\fill}$\Box$
\end{proof}

\begin{claim}\label{claim2}
$\{X^{g}_{\varepsilon,\delta}\}_{\delta\in(0,1)}$ converges to $X^g_\varepsilon$ in $C([0,T];F^*_{1,2})$ as $\delta\rightarrow0$, and $X^g_\varepsilon\in L^2([0,T];L^2(\mu))$.
\end{claim}

\begin{proof}
By the chain rule, we get that for all $\delta, \delta'\in(0,1)$, and $t\in [0,T]$,
\begin{eqnarray}\label{c21}
&&\!\!\!\!\!\!\!\!\|X^g_{\varepsilon,\delta}(t)-X^g_{\varepsilon,\delta'}(t)\|^2_{F^*_{1,2,\varepsilon}}\nonumber\\
&&\!\!\!\!\!\!\!\!+2\int_0^t\big\langle\Psi(X^g_{\varepsilon,\delta}(s))-\Psi(X^g_{\varepsilon,\delta'}(s))+\delta X^g_{\varepsilon,\delta}(s)-\delta'X^g_{\varepsilon,\delta'}(s),X^g_{\varepsilon,\delta}(s)-X^g_{\varepsilon,\delta'}(s)\big\rangle_2ds\nonumber\\
=&&\!\!\!\!\!\!\!\!2\int_0^t\big\langle X^g_{\varepsilon,\delta}(s)-X^g_{\varepsilon,\delta'}(s),\nonumber\\
&&\!\!\!\!\!\!\!\!~~~~~~~~~~~~~~~\int_Z\big(f(s,X^g_{\varepsilon,\delta}(s),z)-f(s,X^g_{\varepsilon,\delta'}(s),z)\big)\cdot(g(s,z)-1)\nu(dz)\big\rangle_{F^*_{1,2,\varepsilon}} ds.
\end{eqnarray}
By \cite[(4.25)]{WZ}, we know that the second term on the left-hand side of \eqref{c21} can be estimated as follows.
\begin{eqnarray}\label{c22}
&&\!\!\!\!\!\!\!\!2\int_0^t\Psi(X^g_{\varepsilon,\delta}(s))-\Psi(X^g_{\varepsilon,\delta'}(s))+\delta X^g_{\varepsilon,\delta}(s)-\delta'X^g_{\varepsilon,\delta'}(s),X^g_{\varepsilon,\delta}(s)-X^g_{\varepsilon,\delta'}(s)\big\rangle_2ds\nonumber\\
\geq&&\!\!\!\!\!\!\!\!2\tilde{\alpha}\int_0^t|\Psi(X^g_{\varepsilon,\delta}(s))-\Psi(X^g_{\varepsilon,\delta'}(s))|_2^2ds\nonumber\\
&&\!\!\!\!\!\!\!\!+2\int_0^t\big\langle \delta X^g_{\varepsilon,\delta}(s)-\delta'X^g_{\varepsilon,\delta'}(s),X^g_{\varepsilon,\delta}(s)-X^g_{\varepsilon,\delta'}(s)\big\rangle_2ds.
\end{eqnarray}
For the right hand-side of \eqref{c21}, by \eqref{f}, we have
\begin{eqnarray}\label{c23}
&&\!\!\!\!\!\!\!\!2\!\!\int_0^t\!\!\!\big\langle X^g_{\varepsilon,\delta}(s)\!\!-\!\!X^g_{\varepsilon,\delta'}(s),\int_Z\!\!\!\big(f(s,X^g_{\varepsilon,\delta}(s),z)\!-\!f(s,X^g_{\varepsilon,\delta'}(s),z)\big)\cdot(g(s,z)\!-\!1)\nu(dz)\big\rangle_{F^*_{1,2,\varepsilon}} \!\!\!\!ds\nonumber\\
\leq&&\!\!\!\!\!\!\!\!2\int_0^t\|X^g_{\varepsilon,\delta}(s)-X^g_{\varepsilon,\delta'}(s)\|_{F^*_{1,2,\varepsilon}}\cdot\nonumber\\
&&\!\!\!\!\!\!\!\!~~~~~~~~~~~~~~\big\|\int_Z\big(f(s,X^g_{\varepsilon,\delta}(s),z)-f(s,X^g_{\varepsilon,\delta'}(s),z)\big)\cdot(g(s,z)-1)\nu(dz)\big\|_{F^*_{1,2,\varepsilon}}ds\nonumber\\
\leq&&\!\!\!\!\!\!\!\!2\int_0^t\|X^g_{\varepsilon,\delta}(s)-X^g_{\varepsilon,\delta'}(s)\|_{F^*_{1,2,\varepsilon}}\cdot\nonumber\\
&&\!\!\!\!\!\!\!\!~~~~~~~~~~~~~~\int_Z\big\|f(s,X^g_{\varepsilon,\delta}(s),z)-f(s,X^g_{\varepsilon,\delta'}(s),z)\big\|_{F^*_{1,2,\varepsilon}}\cdot|g(s,z)-1|\nu(dz)ds\nonumber\\
\leq&&\!\!\!\!\!\!\!\!2\int_0^t\|X^g_{\varepsilon,\delta}(s)-X^g_{\varepsilon,\delta'}(s)\|_{F^*_{1,2,\varepsilon}}\cdot\int_Z\frac{l_1(s,z)}{\sqrt{\varepsilon}}\big\|X^g_{\varepsilon,\delta}(s)-X^g_{\varepsilon,\delta}(s)\big\|_{F^*_{1,2,\varepsilon}}\cdot|g(s,z)-1|\nu(dz)ds\nonumber\\
=&&\!\!\!\!\!\!\!\!\frac{2}{\sqrt{\varepsilon}}\int_0^t\|X^g_{\varepsilon,\delta}(s)-X^g_{\varepsilon,\delta'}(s)\|^2_{F^*_{1,2,\varepsilon}}\cdot h_1(s)ds.
\end{eqnarray}
Plugging \eqref{c22} and \eqref{c23} into \eqref{c21}, we get
\begin{eqnarray*}
&&\!\!\!\!\!\!\!\!\sup_{s\in[0,t]}\|X^g_{\varepsilon,\delta}(s)-X^g_{\varepsilon,\delta'}(s)\|^2_{F^*_{1,2,\varepsilon}}+2\tilde{\alpha}\int_0^t|\Psi(X^g_{\varepsilon,\delta}(s))-\Psi(X^g_{\varepsilon,\delta'}(s))|_2^2ds\nonumber\\
\leq&&\!\!\!\!\!\!\!\!\frac{2}{\sqrt{\varepsilon}}\int_0^th_1(s)\|X^g_{\varepsilon,\delta}(s)-X^g_{\varepsilon,\delta'}(s)\|^2_{F^*_{1,2,\varepsilon}}ds+4(\delta+\delta')\int_0^t|X^g_{\varepsilon,\delta}(s)|_2^2+|X^g_{\varepsilon,\delta'}(s)|_2^2ds.
\end{eqnarray*}
By Gronwall's inequality and Claim \ref{claim1}(see (\ref{eq zhai 3})), we know that there exists some constant $C_{l_1,l_2,T,N,\varepsilon}\in(0,\infty)$,
which depends on $l_1, l_2, T, N, \varepsilon$,
but is independent of $\delta$ and $\delta'$,
such that
\begin{eqnarray*}
&&\!\!\!\!\!\!\!\!\sup_{s\in[0,t]}\|X^g_{\varepsilon,\delta}(s)-X^g_{\varepsilon,\delta'}(s)\|^2_{F^*_{1,2,\varepsilon}}+2\tilde{\alpha}\int_0^t|\Psi(X^g_{\varepsilon,\delta}(s))-\Psi(X^g_{\varepsilon,\delta'}(s))|_2^2ds\nonumber\\
\leq&&\!\!\!\!\!\!\!\!e^{\frac{2}{\sqrt{\varepsilon}}\int_0^th_1(s)ds}\cdot 4(\delta+\delta')\int_0^t|X^g_{\varepsilon,\delta}(s)|_2^2+|X^g_{\varepsilon,\delta'}(s)|_2^2 ds\nonumber\\
\leq&&\!\!\!\!\!\!\!\!e^{\frac{2}{\sqrt{\varepsilon}}C_{l_1,N}}\cdot4(\delta+\delta')\cdot 2T\big(2|x|_2^2+8C_{l_3,N}^2+4T\big)e^{8C_{l_3,N}^2+4T}\nonumber\\
:=&&\!\!\!\!\!\!\!\!C_{l_1,l_3,T,N,\varepsilon}(\delta+\delta')(|x|_2^2+1).
\end{eqnarray*}
Therefore, there exists an element $X^g_\varepsilon\in C([0,T];F^*_{1,2})$ such that $\{X^g_{\varepsilon,\delta}\}_{\delta\in(0,1)}$ converges to $X^g_\varepsilon$ in $C([0,T];F^*_{1,2})$ as $\delta\rightarrow0$. In addition, by Claim \ref{claim1}, we know that $X^g_\varepsilon\in L^2([0,T];L^2(\mu))$.\hspace{\fill}$\Box$
\end{proof}

\begin{claim}\label{claim3}
$X^g_\varepsilon$ satisfies \eqref{eq:3}.
\end{claim}
\begin{proof}
From Claim \ref{claim2}, we know that
\begin{eqnarray}\label{c32}
X^g_{\varepsilon,\delta}\rightarrow X^g_\varepsilon,~~\text{in}~~ C([0,T];F^*_{1,2})~~\text{as}~~\delta\rightarrow0.
\end{eqnarray}
By \textbf{(H2)(i)}, H\"{o}lder's inequality and \eqref{c32}, we have
\begin{eqnarray}\label{c31}
&&\hspace{-1.3truecm}\sup_{t\in[0,T]}\Big\|\int_0^t\int_Zf(s,X^g_{\varepsilon,\delta}(s),z)(g(s,z)-1)\nu(dz)ds-\int_0^t\int_Zf(s,X^g_{\varepsilon}(s),z)(g(s,z)-1)\nu(dz)ds\Big\|_{F^*_{1,2}}\nonumber\\
\leq&&\!\!\!\!\!\!\!\!\int_0^T\int_Z\big\|f(s,X^g_{\varepsilon,\delta}(s),z)-f(s,X^g_{\varepsilon}(s),z)\big\|_{F^*_{1,2}}|g(s,z)-1|\nu(dz)ds\nonumber\\
\leq&&\!\!\!\!\!\!\!\!\int_0^T\int_Zl_1(s,z)|g(s,z)-1|\cdot\|X^g_{\varepsilon,\delta}(s)-X^g_{\varepsilon}(s)\|_{F^*_{1,2}}\nu(dz)ds\nonumber\\
=&&\!\!\!\!\!\!\!\!\int_0^Th_1(s)\|X^g_{\varepsilon,\delta}(s)-X^g_{\varepsilon}(s)\|_{F^*_{1,2}}ds\nonumber\\
\leq&&\!\!\!\!\!\!\!\!\sup_{s\in[0,T]}\|X^g_{\varepsilon,\delta}(s)-X^g_{\varepsilon}(s)\|_{F^*_{1,2}}\cdot\int_0^Th_1(s)ds\nonumber\\
\leq&&\!\!\!\!\!\!\!\!C_{l_1,N}\Big(\sup_{s\in[0,T]}\|X^g_{\varepsilon,\delta}(s)-X^g_{\varepsilon}(s)\|^2_{F^*_{1,2}}\Big)^{\frac{1}{2}}\nonumber\\
&&\!\!\!\!\!\!\!\!\longrightarrow 0,~~\text{as}~~\delta\longrightarrow0,
\end{eqnarray}
which means
\begin{eqnarray*}
&&\!\!\!\!\!\!\!\!\int_0^\cdot\int_Zf(s,X^g_{\varepsilon,\delta}(s),z)(g(s,z)-1)\nu(dz)\nonumber\\
&&\!\!\!\!\!\!\!\!\longrightarrow\int_0^\cdot\int_Zf(s,X^g_{\varepsilon}(s),z)(g(s,z)-1)\nu(dz),~~\text{as}~~\delta\longrightarrow0,~~\text{in}~~C([0,T];F^*_{1,2}).
\end{eqnarray*}
We can rewrite \eqref{eq:4} as the following:
\begin{eqnarray}\label{c33}
&&\!\!\!\!\!\!\!\!(L-\varepsilon)\int_0^\cdot\big(\Psi(X^{g}_{\varepsilon,\delta}(s))+\delta X^{g}_{\varepsilon,\delta}(s)\big)ds\nonumber\\
=&&\!\!\!\!\!\!\!\!-x+X^{g}_{\varepsilon,\delta}(\cdot)-\int_0^\cdot\int_{Z}f(s,X^{g}_{\varepsilon,\delta}(s),z)\big(g(s,z)-1\big)\nu(dz)ds,\ \text{in}\ [0,T].
\end{eqnarray}
Then, by \eqref{c32}-\eqref{c33}, we know that as $\delta\rightarrow0$,
$$\int_0^\cdot\big(\Psi(X^{g}_{\varepsilon,\delta}(s))+\delta X^{g}_{\varepsilon,\delta}(s)\big)ds$$
converges to some element in $C([0,T];F_{1,2})$.
Using the similar arguments as in \cite[Claim 4.3]{WZ21},
we know that as $\delta\rightarrow0$,
\begin{eqnarray*}
&&\!\!\!\!\!\!\!\!\Psi(X^g_{\varepsilon,\delta}(\cdot))+\delta X^g_{\varepsilon,\delta}(\cdot)\longrightarrow\Psi(X^g_{\varepsilon}(\cdot))~\text{weakly~in}~~L^2([0,T];L^2(\mu)),
\end{eqnarray*}
and furthermore, $\int_0^\cdot\Psi(X^g_\varepsilon(s))ds\in C([0,T];F_{1,2})$.
This implies Claim \ref{claim3}.\hspace{\fill}$\Box$
\end{proof}

By lower semicontinuity of the norm and Claim \ref{claim1}, we also know that \eqref{p3} holds, i.e.,
\begin{eqnarray}\label{est2}
\sup_{g\in S^N}\sup_{t\in[0,T]}|X^{g}_{\varepsilon}(t)|_2^2\leq C_{N,T}(|x|_2^2+1).
\end{eqnarray}

\vspace{2mm}
\textbf{Uniqueness}
\vspace{2mm}

Assume $X^g_{\varepsilon,1}$ and $X^g_{\varepsilon,2}$ are two solutions to \eqref{eq:3}. Applying the chain rule to $\|X^g_{\varepsilon,1}-X^g_{\varepsilon,2}\|^2_{F^*_{1,2,\varepsilon}}$,
\begin{eqnarray}\label{c34}
&&\!\!\!\!\!\!\!\!\|X^g_{\varepsilon,1}(t)-X^g_{\varepsilon,2}(t)\|^2_{F^*_{1,2,\varepsilon}}\nonumber\\
&&\!\!\!\!\!\!\!\!+2\int_0^t\big\langle \Psi(X^g_{\varepsilon,1}(s))-\Psi(X^g_{\varepsilon,2}(s)),X^g_{\varepsilon,1}(s)-X^g_{\varepsilon,2}(s)\big\rangle_2ds\nonumber\\
=&&\!\!\!\!\!\!\!\!2\int_0^t\big\langle \int_Z\big(f(s,X^g_{\varepsilon,1}(s),z)-f(s,X^g_{\varepsilon,2}(s),z)\big)\nu(dz),X^g_{\varepsilon,1}(s)-X^g_{\varepsilon,2}(s)\big\rangle_{F^*_{1,2,\varepsilon}}ds.
\end{eqnarray}
Since $\Psi$ is Lipschitz, we have
\begin{eqnarray}\label{psii}
\big(\Psi(r)-\Psi(r')\big)(r-r')\geq \tilde{\alpha}|\Psi(r)-\Psi(r')|^2,\ \forall r, r'\in\Bbb{R}.
\end{eqnarray}
By \eqref{triple2}, \eqref{psii}, and \eqref{f}, we know that
\begin{eqnarray}\label{c35}
&&\!\!\!\!\!\!\!\!\|X^g_{\varepsilon,1}(t)-X^g_{\varepsilon,2}(t)\|^2_{F^*_{1,2,\varepsilon}}+2\tilde{\alpha}\int_0^t|\Psi(X^g_{\varepsilon,1}(s))-\Psi(X^g_{\varepsilon,2}(s))|_2^2ds\nonumber\\
\leq&&\!\!\!\!\!\!\!\!\frac{2}{\sqrt{\varepsilon}}\int_0^th_1(s)\|X^g_{\varepsilon,1}(s)-X^g_{\varepsilon,2}(s)\|^2_{F^*_{1,2,\varepsilon}}ds.
\end{eqnarray}
Since $h_1\in L^1([0,T];\Bbb{R}^+)$, 
by Gronwall's inequality, we can get $X^g_{\varepsilon,1}=X^g_{\varepsilon,2}$, which implies the uniqueness. This completes the proof of Proposition \ref{prop2}. \hspace{\fill}$\Box$

\vspace{3mm}
\textbf{Continuation of proof of Theorem \ref{Th2}}
\vspace{3mm}

The idea is to prove that the sequence $\{X^g_{\varepsilon}\}_{\varepsilon\in(0,1)}$ converges to the solution of \eqref{eq:2} as $\varepsilon\rightarrow0$.
Rewrite \eqref{eq:3} as follows.
\begin{eqnarray*}
dX^g_\varepsilon(t)+(1-L)\Psi(X^g_\varepsilon(t))dt=(1-\varepsilon)\Psi(X^g_\varepsilon(t))dt+\int_Zf(t,X^g_\varepsilon(t),z)(g(t,z)-1)\nu(dz)dt.
\end{eqnarray*}
Applying the chain rule in $F^*_{1,2}$, by \eqref{triple2}, we obtain
\begin{eqnarray}\label{t21}
&&\!\!\!\!\!\!\!\!\frac{1}{2}\|X^g_\varepsilon(t)\|^2_{F^*_{1,2}}+\int_0^t\langle \Psi(X^g_\varepsilon(s)),X^g_\varepsilon(s)\rangle_2ds\nonumber\\
=&&\!\!\!\!\!\!\!\!\frac{1}{2}\|x\|^2_{F^*_{1,2}}+(1-\varepsilon)\int_0^t\langle \Psi(X^g_\varepsilon(s)),X^g_\varepsilon(s)\rangle_{F^*_{1,2}}ds\nonumber\\
&&\!\!\!\!\!\!\!\!+\int_0^t\big\langle\int_Zf(s,X^g_\varepsilon(s),z)(g(s,z)-1)\nu(dz),X^g_\varepsilon(s)\big\rangle_{F^*_{1,2}} ds.
\end{eqnarray}
By \eqref{psii} and $\Psi(0)=0$, we know that
\begin{eqnarray}\label{t22}
\int_0^t\langle \Psi(X^g_\varepsilon(s)),X^g_\varepsilon(s)\rangle_2ds\geq\tilde{\alpha}\int_0^t|\Psi(X^g_\varepsilon(s))|_2^2ds.
\end{eqnarray}
Substituting \eqref{t22} into \eqref{t21}, we get that by \textbf{(H2)(ii)}, by Young's inequality, and since $L^2(\mu)\subset F^*_{1,2}$ densely and continuously, we have
\begin{eqnarray*}
&&\!\!\!\!\!\!\!\!\frac{1}{2}\|X^g_\varepsilon(t)\|^2_{F^*_{1,2}}+\tilde{\alpha}\int_0^t|\Psi(X^g_\varepsilon(s))|_2^2ds\nonumber\\
\leq&&\!\!\!\!\!\!\!\!\frac{1}{2}\|x\|^2_{F^*_{1,2}}+\int_0^t\|\Psi(X^g_\varepsilon(s))\|_{F^*_{1,2}}\cdot\|X^g_\varepsilon(s)\|_{F^*_{1,2}}ds\nonumber\\
&&\!\!\!\!\!\!\!\!+\int_0^t\int_Zl_2(s,z)(\|X^g_\varepsilon(s)\|_{F^*_{1,2}}+1)|g(s,z)-1|\nu(dz)\cdot\|X^g_\varepsilon(s)\|_{F^*_{1,2}}ds\nonumber\\
=&&\!\!\!\!\!\!\!\!\frac{1}{2}\|x\|^2_{F^*_{1,2}}+\int_0^t\|\Psi(X^g_\varepsilon(s))\|_{F^*_{1,2}}\cdot\|X^g_\varepsilon(s)\|_{F^*_{1,2}}ds+\int_0^th_2(s)(\|X^g_\varepsilon(s)\|_{F^*_{1,2}}+1)\cdot\|X^g_\varepsilon(s)\|_{F^*_{1,2}}ds\nonumber\\
\leq&&\!\!\!\!\!\!\!\!\frac{1}{2}\|x\|^2_{F^*_{1,2}}+\int_0^t|\Psi(X^g_\varepsilon(s))|_2\cdot\|X^g_\varepsilon(s)\|_{F^*_{1,2}}ds+\int_0^th_2(s)\cdot 2(\|X^g_\varepsilon(s)\|^2_{F^*_{1,2}}+1)ds\nonumber\\
\leq&&\!\!\!\!\!\!\!\!\frac{1}{2}\|x\|^2_{F^*_{1,2}}+\int_0^t\frac{\tilde{\alpha}}{2}|\Psi(X^g_\varepsilon(s))|^2_2ds+\int_0^t\frac{1}{2\tilde{\alpha}}\|X^g_\varepsilon(s)\|^2_{F^*_{1,2}}ds\nonumber\\
&&\!\!\!\!\!\!\!\!+\int_0^t2h_2(s)\|X^g_\varepsilon(s)\|^2_{F^*_{1,2}}ds+\int_0^t2h_2(s)ds\nonumber\\
=&&\!\!\!\!\!\!\!\!\frac{1}{2}\|x\|^2_{F^*_{1,2}}+\int_0^t\frac{\tilde{\alpha}}{2}|\Psi(X^g_\varepsilon(s))|^2_2ds+\int_0^t\Big(\frac{1}{2\tilde{\alpha}}+2h_2(s)\Big)\|X^g_\varepsilon(s)\|^2_{F^*_{1,2}}ds+\int_0^t2h_2(s)ds,
\end{eqnarray*}
this yields
\begin{eqnarray}
&&\!\!\!\!\!\!\!\!\|X^g_\varepsilon(t)\|^2_{F^*_{1,2}}+\tilde{\alpha}\int_0^t|\Psi(X^g_\varepsilon(s))|_2^2ds\nonumber\\
\leq&&\!\!\!\!\!\!\!\!\|x\|^2_{F^*_{1,2}}+\int_0^t\Big(\frac{1}{\tilde{\alpha}}+4h_2(s)\Big)\|X^g_\varepsilon(s)\|^2_{F^*_{1,2}}ds+\int_0^t4h_2(s)ds.
\end{eqnarray}
By Gronwall's lemma, for any $t\in[0,T]$,
\begin{eqnarray}\label{est1}
&&\!\!\!\!\!\!\!\!\|X^g_\varepsilon(t)\|^2_{F^*_{1,2}}+\tilde{\alpha}\int_0^t|\Psi(X^g_\varepsilon(s))|_2^2ds\nonumber\\
\leq&&\!\!\!\!\!\!\!\!\big(\|x\|^2_{F^*_{1,2}}+\int_0^t4h_2(s)ds\big)\cdot \exp\Big\{\int_0^t\frac{1}{\tilde{\alpha}}+4h_2(s)ds\Big\}\nonumber\\
\leq&&\!\!\!\!\!\!\!\!\big(\|x\|^2_{F^*_{1,2}}+4C_{l_2,N}\big)\cdot e^{\frac{T}{\tilde{\alpha}}+4C_{l_2,N}}\nonumber\\
:=&&\!\!\!\!\!\!\!\!C_{l_2,N,T}(\|x\|^2_{F^*_{1,2}}+1).
\end{eqnarray}

Now, let us prove the convergence of $\{X^g_\varepsilon\}_{\varepsilon\in(0,1)}$.
Applying the chain rule to $\|X^g_\varepsilon(t)-X^g_{\varepsilon'}(t)\|^2_{F^*_{1,2}}$, $\varepsilon, \varepsilon'\in(0,1)$, by \eqref{triple2}, we get that, for all $t\in[0,T]$,
\begin{eqnarray}\label{t23}
&&\!\!\!\!\!\!\!\!\|X^g_\varepsilon(t)-X^g_{\varepsilon'}(t)\|^2_{F^*_{1,2}}+2\int_0^t\big\langle\Psi(X^g_\varepsilon(s))-\Psi(X^g_{\varepsilon'}(s)),X^g_\varepsilon(s)-X^g_{\varepsilon'}(s)\big\rangle_2ds\nonumber\\
=&&\!\!\!\!\!\!\!\!2\int_0^t\big\langle\Psi(X^g_\varepsilon(s))-\Psi(X^g_{\varepsilon'}(s)),X^g_\varepsilon(s)-X^g_{\varepsilon'}(s)\big\rangle_{F^*_{1,2}}ds\nonumber\\
&&\!\!\!\!\!\!\!\!-2\int_0^t\big\langle\varepsilon\Psi(X^g_\varepsilon(s))-\varepsilon'\Psi(X^g_{\varepsilon'}(s)),X^g_\varepsilon(s)-X^g_{\varepsilon'}(s)\big\rangle_{F^*_{1,2}}ds\nonumber\\
&&\!\!\!\!\!\!\!\!+2\int_0^t\Big\langle\int_Z\big(f(s,X^g_\varepsilon(s),z)-f(s,X^g_{\varepsilon'}(s),z)\big)(g(s,z)-1)\nu(dz),\nonumber\\
&&\!\!\!\!\!\!\!\!~~~~~~~~~~~~~~~~~~~~~~~~~~~~~~~~~~~~~~~~~~~~~~~~~~~~~~~~~~~~~~~~~~~~~X^g_\varepsilon(s)-X^g_{\varepsilon'}(s)\Big\rangle_{F^*_{1,2}}ds.
\end{eqnarray}
By \eqref{psii}, we know that
\begin{eqnarray}\label{t24}
&&\!\!\!\!\!\!\!\!2\int_0^t\big\langle\Psi(X^g_\varepsilon(s))-\Psi(X^g_{\varepsilon'}(s)),X^g_\varepsilon(s)-X^g_{\varepsilon'}(s)\big\rangle_2ds\nonumber\\
\geq&&\!\!\!\!\!\!\!\!2\tilde{\alpha}\int_0^t|\Psi(X^g_\varepsilon(s))-\Psi(X^g_{\varepsilon'}(s))|_2^2ds.
\end{eqnarray}
Since $L^2(\mu)$ densely embeds into $F^*_{1,2}$, the first term in the right hand-side of \eqref{t23} is dominated by
\begin{eqnarray}\label{t25}
2\int_0^t|\Psi(X^g_\varepsilon(s))-\Psi(X^g_{\varepsilon'}(s))|_2\cdot\|X^g_\varepsilon(s)-X^g_{\varepsilon'}(s)\|_{F^*_{1,2}}ds.
\end{eqnarray}
Similarly, the second term in the right hand-side of \eqref{t23} is dominated by
\begin{eqnarray}\label{t26}
2\int_0^t\big(\varepsilon|\Psi(X^g_\varepsilon(s))|_2+\varepsilon'|\Psi(X^g_{\varepsilon'}(s))|_2\big)\cdot\|X^g_\varepsilon(s)-X^g_{\varepsilon'}(s)\|_{F^*_{1,2}}ds.
\end{eqnarray}
By \textbf{(H2)(i)} and \eqref{h1}, the third term in the right hand-side of \eqref{t23} is dominated by
\begin{eqnarray}\label{t27}
2\int_0^t h_1(s)\|X^g_\varepsilon(s)-X^g_{\varepsilon'}(s)\|^2_{F^*_{1,2}}ds.
\end{eqnarray}

Taking \eqref{t24}-\eqref{t27} into \eqref{t23}, by Young's inequality, we get
\begin{eqnarray*}
&&\!\!\!\!\!\!\!\!\|X^g_\varepsilon(t)-X^g_{\varepsilon'}(t)\|^2_{F^*_{1,2}}+2\tilde{\alpha}\int_0^t|\Psi(X^g_\varepsilon(s))-\Psi(X^g_{\varepsilon'}(s))|_2^2ds\nonumber\\
\leq&&\!\!\!\!\!\!\!\!2\int_0^t|\Psi(X^g_\varepsilon(s))-\Psi(X^g_{\varepsilon'}(s))|_2\cdot \|X^g_\varepsilon(s)-X^g_{\varepsilon'}(s)\|_{F^*_{1,2}}ds\nonumber\\
&&\!\!\!\!\!\!\!\!+2\int_0^t\big(\varepsilon|\Psi(X^g_\varepsilon(s))|_2+\varepsilon'|\Psi(X^g_{\varepsilon'}(s))|_2\big)\cdot\|X^g_\varepsilon(s)-X^g_{\varepsilon'}(s)\|_{F^*_{1,2}}ds\nonumber\\
&&\!\!\!\!\!\!\!\!+2\int_0^t h_1(s)\|X^g_\varepsilon(s)-X^g_{\varepsilon'}(s)\|^2_{F^*_{1,2}}ds\nonumber\\
\leq&&\!\!\!\!\!\!\!\!\tilde{\alpha}\int_0^t|\Psi(X^g_\varepsilon(s))-\Psi(X^g_{\varepsilon'}(s))|_2^2ds+\frac{1}{\tilde{\alpha}}\int_0^t\|X^g_\varepsilon(s)-X^g_{\varepsilon'}(s)\|^2_{F^*_{1,2}}ds\nonumber\\
&&\!\!\!\!\!\!\!\!+2\int_0^t\varepsilon^2|\Psi(X^g_\varepsilon(s))|^2_2+\varepsilon'^2|\Psi(X^g_{\varepsilon'}(s))|^2_2ds+\int_0^t\|X^g_\varepsilon(s)-X^g_{\varepsilon'}(s)\|^2_{F^*_{1,2}}ds\nonumber\\
&&\!\!\!\!\!\!\!\!+2\int_0^th_1(s)\|X^g_\varepsilon(s)-X^g_{\varepsilon'}(s)\|^2_{F^*_{1,2}}ds\nonumber\\
\leq&&\!\!\!\!\!\!\!\!\tilde{\alpha}\int_0^t|\Psi(X^g_\varepsilon(s))-\Psi(X^g_{\varepsilon'}(s))|_2^2ds+\int_0^t\big(\frac{1}{\tilde{\alpha}}+1+2h_1(s)\big)\cdot\|X^g_\varepsilon(s)-X^g_{\varepsilon'}(s)\|^2_{F^*_{1,2}}ds\nonumber\\
&&\!\!\!\!\!\!\!\!+2(\varepsilon^2+\varepsilon'^2)\int_0^t|\Psi(X^g_\varepsilon(s))|^2_2+|\Psi(X^g_{\varepsilon'}(s))|^2_2ds.
\end{eqnarray*}
This yields,
\begin{eqnarray*}
&&\!\!\!\!\!\!\!\!\sup_{s\in[0,t]}\|X^g_\varepsilon(s)-X^g_{\varepsilon'}(s)\|^2_{F^*_{1,2}}+\tilde{\alpha}\int_0^t|\Psi(X^g_\varepsilon(s))-\Psi(X^g_{\varepsilon'}(s))|_2^2ds\nonumber\\
\leq&&\!\!\!\!\!\!\!\!\int_0^t(\frac{1}{\tilde{\alpha}}+1+2h_1(s))\cdot\|X^g_\varepsilon(s)-X^g_{\varepsilon'}(s)\|^2_{F^*_{1,2}}ds\nonumber\\
&&\!\!\!\!\!\!\!\!+2(\varepsilon^2+\varepsilon'^2)\int_0^t|\Psi(X^g_\varepsilon(s))|^2_2+|\Psi(X^g_{\varepsilon'}(s))|^2_2ds.
\end{eqnarray*}
Since \eqref{est1} holds, by Gronwall's inequality, we know that there exists a positive constant $C_{l_1,l_2,N,T}\in(0,\infty)$ which depends on $l_1, l_2, N, T$, but is independent of $\varepsilon, \varepsilon'$, such that
\begin{eqnarray*}
&&\!\!\!\!\!\!\!\!\sup_{s\in[0,T]}\|X^g_\varepsilon(s)-X^g_{\varepsilon'}(s)\|^2_{F^*_{1,2}}+\tilde{\alpha}\int_0^T|\Psi(X^g_\varepsilon(s))-\Psi(X^g_{\varepsilon'}(s))|_2^2ds\nonumber\\
\leq&&\!\!\!\!\!\!\!\!2(\varepsilon^2+\varepsilon'^2)\int_0^T|\Psi(X^g_\varepsilon(s))|^2_2+|\Psi(X^g_{\varepsilon'}(s))|^2_2ds\cdot e^{\int_0^T\frac{1}{\tilde{\alpha}}+1+2h_1(s)ds}\nonumber\\
\leq&&\!\!\!\!\!\!\!\!4(\varepsilon^2+\varepsilon'^2)\cdot\frac{1}{\tilde{\alpha}}C_{l_2,N,T}(\|x\|^2_{F^*_{1,2}}+1)\cdot e^{\frac{T}{\tilde{\alpha}}+T+2C_{l_1,N}}\nonumber\\
:=&&\!\!\!\!\!\!\!\!C_{l_1,l_2,N,T}(\varepsilon^2+\varepsilon'^2)(\|x\|^2_{F^*_{1,2}}+1).
\end{eqnarray*}
Hence, there exists a function $X^g\in C([0,T];F^*_{1,2})$ such that $X^g_\varepsilon\rightarrow X^g$ in $C([0,T];F^*_{1,2})$ as $\varepsilon\rightarrow0$.

Next, let us prove that $X^g$ satisfies \eqref{eq:2}. Since
\begin{eqnarray*}
X^g_\varepsilon\longrightarrow X^g~~\text{in}~~C([0,T];F^*_{1,2})~~\text{as}~~\varepsilon\rightarrow0,
\end{eqnarray*}
using similar arguments as in Claim \ref{claim3} and \cite[Claim 4.3]{WZ21}, we have
\begin{eqnarray*}\label{t1}
\int_0^\cdot\int_Zf(s,X^g_\varepsilon(s),z)(g(s,z)-1)\nu(dz)ds\rightarrow\int_0^\cdot\int_Zf(s,X^g(s),z)(g(s,z)-1)\nu(dz)ds
\end{eqnarray*}
in $C([0,T];F^*_{1,2})$ as $\varepsilon\rightarrow0$,
\begin{eqnarray}\label{t2}
\Psi(X^g_\varepsilon(\cdot))\rightarrow\Psi(X^g(\cdot))~\text{weakly~in}~~L^2([0,T];L^2(\mu))~~\text{as}~~\varepsilon\rightarrow0,
\end{eqnarray}
and furthermore, $\int_0^\cdot\Psi(X^g(s))ds\in C([0,T];F_{1,2})$.

Hence $X^g$ satisfies \eqref{eq:2}. By lower semicontinuity of the norm and \eqref{est2}, we know that \eqref{defi4} holds. This completes the proof of the existence of solutions in Theorem \ref{Th2}.

\vspace{3mm}
\textbf{Uniqueness}
\vspace{3mm}

Assume that $X^g_1$ and $X^g_2$ are two solutions to \eqref{eq:2}, we know that
\begin{eqnarray}\label{t26}
&&\!\!\!\!\!\!\!\!X^g_1(t)-X^g_2(t)-L\int_0^t\Psi(X^g_1(s))-\Psi(X^g_2(s))ds\nonumber\\
=&&\!\!\!\!\!\!\!\!\int_0^t\int_Z\big(f(s,X^g_1(s),z)-f(s,X^g_2(s),z)\big)(g(s,z)-1)\nu(dz)ds,~~\forall~t\in[0,T].
\end{eqnarray}
Rewrite \eqref{t26} as
\begin{eqnarray*}
&&\!\!\!\!\!\!\!\!X^g_1(t)-X^g_2(t)+(1-L)\int_0^t\Psi(X^g_1(s))-\Psi(X^g_2(s))ds\nonumber\\
=&&\!\!\!\!\!\!\!\!\int_0^t\Psi(X^g_1(s))-\Psi(X^g_2(s))ds\nonumber\\
&&\!\!\!\!\!\!\!\!+\int_0^t\int_Z\big(f(s,X^g_1(s),z)-f(s,X^g_2(s),z)\big)(g(s,z)-1)\nu(dz)ds,~~\forall~t\in[0,T].
\end{eqnarray*}
Applying the chain rule to $\|X^g_1(t)-X^g_2(t)\|^2_{F^*_{1,2}}$, by \eqref{triple2}, we know that,
\begin{eqnarray*}
&&\!\!\!\!\!\!\!\!\|X^g_1(t)-X^g_2(t)\|^2_{F^*_{1,2}}+2\int_0^t\langle \Psi(X^g_1(s))-\Psi(X^g_2(s)), X^g_1(s)-X^g_2(s)\rangle_2ds\nonumber\\
=&&\!\!\!\!\!\!\!\!2\int_0^t\langle \Psi(X^g_1(s))-\Psi(X^g_2(s)), X^g_1(s)-X^g_2(s)\rangle_{F^*_{1,2}}ds\nonumber\\
&&\!\!\!\!\!\!\!\!+2\int_0^t\big\langle\int_Z\big(f(s,X^g_1(s),z)-f(s,X^g_2(s),z)\big)(g(s,z)-1)\nu(dz),X^g_1(s)-X^g_2(s)\big\rangle_{F^*_{1,2}}ds.
\end{eqnarray*}
By \eqref{psii}, $\Psi(0)=0$, and \textbf{(H2)(i)}, we have
\begin{eqnarray*}
&&\!\!\!\!\!\!\!\!\|X^g_1(t)-X^g_2(t)\|^2_{F^*_{1,2}}+2\tilde{\alpha}\int_0^t|\Psi(X^g_1(s))-\Psi(X^g_2(s))|_2^2ds\nonumber\\
\leq&&\!\!\!\!\!\!\!\!2\int_0^t\|\Psi(X^g_1(s))-\Psi(X^g_2(s))\|_{F^*_{1,2}}\cdot\|X^g_1(s)-X^g_2(s)\|_{F^*_{1,2}}ds\nonumber\\
&&\!\!\!\!\!\!\!\!+2\int_0^th_1(s)\cdot\|X^g_1(s)-X^g_2(s)\|^2_{F^*_{1,2}}ds.
\end{eqnarray*}
Since $L^2(\mu)\subset F^*_{1,2}$ continuously and densely, using Young's inequality, we obtain that
\begin{eqnarray*}
&&\!\!\!\!\!\!\!\!\|X^g_1(t)-X^g_2(t)\|^2_{F^*_{1,2}}+2\tilde{\alpha}\int_0^t|\Psi(X^g_1(s))-\Psi(X^g_2(s))|_2^2ds\nonumber\\
\leq&&\!\!\!\!\!\!\!\!2\tilde{\alpha}\int_0^t|\Psi(X^g_1(s))-\Psi(X^g_2(s))|_2^2ds+\int_0^t\big(\frac{1}{2\tilde{\alpha}}+2h_1(s)\big)\|X^g_1(s)-X^g_2(s)\|^2_{F^*_{1,2}}ds,
\end{eqnarray*}
which yields,
\begin{eqnarray*}
\|X^g_1(t)-X^g_2(t)\|^2_{F^*_{1,2}}\leq\int_0^t\big(\frac{1}{2\tilde{\alpha}}+2h_1(s)\big)\|X^g_1(s)-X^g_2(s)\|^2_{F^*_{1,2}}ds.
\end{eqnarray*}
Since $h_1\in L^1([0,T];\Bbb{R}^+)$, by Gronwall's lemma, we get $X^g_1=X^g_2$. This completes the proof of Theorem \ref{Th2}. \hspace{\fill}$\Box$

\section{Large deviations---Proof of Theorem \ref{Th1}}\label{Section5}
\setcounter{equation}{0}
 \setcounter{definition}{0}

 \begin{proof}{\bf Proof of Theorem \ref{Th1}}

Recall $S=\bigcup_{N=1}^\infty S^N$. From Theorem \ref{Th2}, we know that there is a measurable mapping
\begin{eqnarray}\label{G0}
\mathcal{G}^0:S\rightarrow D([0,T];F^*_{1,2})
\end{eqnarray}
such that $\mathcal{G}^0(g):=X^g$, where $X^g$ is the unique solution to \eqref{eq:2}.

 Set
 \begin{eqnarray*}
 {\mathcal{A}}^N:=\{\varphi\in {\mathcal{A}}~\text{and}~\varphi(\omega)\in S^N, {\Bbb{P}}\text{-}a.s.\}.
\end{eqnarray*}

Let $\{K_n\subset Z, n=1,2,...\}$ be an increasing sequence of compact sets of $Z$ such that $\bigcup_{n=1}^\infty K_n=Z$. For each $n$, let $K^c_n=Z\setminus K_n$ and
\begin{eqnarray*}
 {\mathcal{A}}_{b,n}=&&\!\!\!\!\!\!\!\!\{\varphi\in {\mathcal{A}}:~\text{for~all}~(t,\omega)\in[0,T]\times {\Omega},~n\geq\varphi(t,x,\omega)\geq\frac{1}{n}~\text{if}~x\in K_n\\
&&\!\!\!\!\!\!\!\!\text{and}~\varphi(t,x,\omega)=1~\text{if}~x\in K^c_n\},
\end{eqnarray*}
and let $ {\mathcal{A}}_b=\bigcup_{n=1}^\infty {\mathcal{A}}_{b,n}$. Define $\widetilde{\mathcal{A}}^N= {\mathcal{A}}^N\cap  {\mathcal{A}}_b$.

From Theorem \ref{th}, \cite[page: 2851, Lemma 4.3]{XZ} and the statement under it (or \cite[Lemma 7.1]{BPZ}, \cite[Theorem 3.8]{LSZZ}), we see that, for any $\epsilon>0$, there is a measurable mapping
\begin{eqnarray}\label{G1}
\mathcal{G}^\epsilon:\mathcal{M}_{FC}(Z_T)\rightarrow D([0,T];F^*_{1,2})
\end{eqnarray}
such that for any $N\in\Bbb{N}$ and $\varphi_\epsilon\in\widetilde{\mathcal{A}}^N$, we have $\mathcal{G}^\epsilon(\epsilon N^{\epsilon^{-1}\varphi_\epsilon}):=X^{\varphi_\epsilon}$, where $X^{\varphi_\epsilon}$ is the unique solution of the following controlled SPDE:
\begin{eqnarray}\label{ldp1}
X^{\varphi_\epsilon}(t)=&&\!\!\!\!\!\!\!\!x+L\int_0^t\Psi(X^{\varphi_\epsilon}(s))ds+\epsilon\int_0^t\int_Zf(s,X^{\varphi_\epsilon}(s-),z)\widetilde{N}^{\epsilon^{-1}\varphi_\epsilon}(dz,ds)\nonumber\\
&&\!\!\!\!\!\!\!\!+\int_0^t\int_Zf(s,X^{\varphi_\epsilon}(s),z)\big(\varphi_\epsilon(s,z)-1\big)\nu(dz)ds.
\end{eqnarray}

\vspace{3mm}

According to \cite[Theorem 4.4]{LSZZ}, in which the conditions are an adaption of the original conditions given in \cite[page:528, Condition 2.2.]{BCD} and \cite[page:736, Condition 4.1.]{BDM2}, Theorem \ref{Th1} is proved once we can prove the following condition holds:

\begin{Condition}\label{Condition}
(a)~~For any $N<\infty$, let $g_n,~n\geq1, g\in S^N$ be such that $g_n\rightarrow g$ as $n\rightarrow\infty$. Then
\begin{eqnarray*}
\mathcal{G}^0(g_n)\rightarrow\mathcal{G}^0(g)~~\text{in}~D([0,T];F^*_{1,2}).
\end{eqnarray*}

(b)~~For any $N<\infty$, any family $\{\varphi_\epsilon\}_{\epsilon>0}\subset\tilde{\mathcal{A}}^N$, and any $\theta>0$,
\begin{eqnarray*}
\lim_{\epsilon\rightarrow0}\Bbb{P}\{\rho(X^{\varphi_\epsilon},Y^{\varphi_\epsilon})>\theta\}=0,
\end{eqnarray*}
where
$$X^{\varphi_\epsilon}:=\mathcal{G}^\epsilon(\epsilon N^{\epsilon^{-1}\varphi_\epsilon}),~~~Y^{\varphi_\epsilon}:=\mathcal{G}^0(\varphi_\epsilon),$$
and
$\rho(\cdot, \cdot)$ stands for the Skorohod metric on the space $D([0,T];F^*_{1,2})$.
\end{Condition}


We will establish (a) and (b) in Subsections \ref{Subsection5.1} and \ref{Subsection5.2}, respectively.

 \end{proof}

\subsection{Proof of (a) in Condition \ref{Condition}}\label{Subsection5.1}

%
%

\begin{proof}{\bf Proof of (a) in Condition \ref{Condition}}

For $N<\infty$, let $g_n,~n\geq1$, $g\in S^N$ be such that $g_n\rightarrow g$ as $n\rightarrow\infty$. By Theorem \ref{Th2}, we know that the following deterministic PDE:
\begin{equation}\label{ld2}
\left\{
  \begin{array}{ll}
    dX^g(t)-L\Psi(X^g(t))dt=\int_Zf(s,X^g(s),z)(g(s,z)-1)\nu(dz)dt, \\
    X^g(0)=x\in L^2(\mu),
  \end{array}
\right.
\end{equation}
has a unique solution $X^g:=\mathcal{G}^0(g)$, and the following equation:
\begin{equation}\label{ld3}
\left\{
  \begin{array}{ll}
    dX^{g_n}(t)-L\Psi(X^{g_n}(t))dt=\int_Zf(s,X^{g_n}(s),z)(g_n(s,z)-1)\nu(dz)dt, \\
    X^{g_n}(0)=x\in L^2(\mu),
  \end{array}
\right.
\end{equation}
has a unique solution $X^{g_n}:=\mathcal{G}^0(g_n)$.

From \eqref{ld2} and \eqref{ld3}, we can get
\begin{eqnarray*}
&&\!\!\!\!\!\!\!\!d(X^{g_n}(t)-X^g(t))+(1-L)\big(\Psi(X^{g_n}(t))-\Psi(X^g(t))\big)dt\nonumber\\
=&&\!\!\!\!\!\!\!\!\big(\Psi(X^{g_n}(t))-\Psi(X^g(t))\big)dt\nonumber\\
&&\!\!\!\!\!\!\!\!+\int_Zf(s,X^g_n(s),z)\big(g_n(s,z)-1\big)-f(s,X^{g}(s),z)\big(g(s,z)-1\big)\nu(dz)dt.
\end{eqnarray*}
Applying the chain rule to $\|X^{g_n}(t)-X^g(t)\|^2_{F^*_{1,2}}$, by \eqref{triple2}, we get
\begin{eqnarray}\label{5.00}
&&\!\!\!\!\!\!\!\!\|X^{g_n}(t)-X^g(t)\|^2_{F^*_{1,2}}+2\int_0^t\big\langle\Psi(X^{g_n}(s))-\Psi(X^g(s)),X^{g_n}(s)-X^g(s)\big\rangle_2 ds\nonumber\\
=&&\!\!\!\!\!\!\!\!2\int_0^t\big\langle\Psi(X^{g_n}(s))-\Psi(X^g(s)),X^{g_n}(s)-X^g(s)\big\rangle_{F^*_{1,2}} ds\nonumber\\
&&\!\!\!\!\!\!\!\!+2\int_0^t\Big\langle\int_Zf(s,X^{g_n}(s),z)\big(g_n(s,z)-1\big)-f(s,X^g(s),z)\big(g(s,z)-1\big)\nu(dz),\nonumber\\
&&\!\!\!\!\!\!\!\!~~~~~~~~~~~~~~~~~~~~~~~~~~~~~~~~~~~~~~~~~~~~~~~~~~~~~~~~~~~~~~~~~~~X^{g_n}(s)-X^g(s)\Big\rangle_{F^*_{1,2}}ds.
\end{eqnarray}
Consider the second term in the left hand-side of \eqref{5.00}.
From \eqref{psii}, we know that
\begin{eqnarray}\label{5.01}
&&\!\!\!\!\!\!\!\!2\int_0^t\big\langle\Psi(X^{g_n}(s))-\Psi(X^g(s)),X^{g_n}(s)-X^g(s)\big\rangle_2 ds\nonumber\\
\geq&&\!\!\!\!\!\!\!\!2\tilde{\alpha}\int_0^t|\Psi(X^{g_n}(s))-\Psi(X^g(s))|_2^2ds.
\end{eqnarray}
For the first term in the right hand-side of \eqref{5.00}, since $L^2(\mu)\subset F^*_{1,2}$ densely, by Young's inequality, we know that
\begin{eqnarray}\label{5.02}
&&\!\!\!\!\!\!\!\!2\int_0^t\big\langle\Psi(X^{g_n}(s))-\Psi(X^g(s)),X^{g_n}(s)-X^g(s)\big\rangle_{F^*_{1,2}} ds\nonumber\\
\leq&&\!\!\!\!\!\!\!\!2\int_0^t|\Psi(X^{g_n}(s))-\Psi(X^g(s))|_2\cdot\|X^{g_n}(s)-X^g(s)\|_{F^*_{1,2}}ds\nonumber\\
\leq&&\!\!\!\!\!\!\!\!2\tilde{\alpha}\int_0^t|\Psi(X^{g_n}(s))-\Psi(X^g(s))|_2^2ds+\frac{1}{2\tilde{\alpha}}\int_0^t\|X^{g_n}(s)-X^g(s)\|^2_{F^*_{1,2}}ds.
\end{eqnarray}
To estimate the second term in the right hand-side of \eqref{5.00}, for simplicity, denote
\begin{eqnarray}\label{5.03}
&&\!\!\!\!\!\!\!\!I_n(t)\nonumber\\
:=&&\!\!\!\!\!\!\!\!2\int_0^t\Big\langle\int_Z\Big(f(s,X^{g_n}(s),z)\big(g_n(s,z)-1\big)-f(s,X^g(s),z)\big(g(s,z)-1\big)\Big)\nu(dz),\nonumber\\
&&\!\!\!\!\!\!\!\!~~~~~~~~~~~~~~~~~~~~~~~~~~~~~~~~~~~~~~~~~~~~~~~~~~~~~~~~~~~~~~~~~~~~~~~~~~X^{g_n}(s)-X^g(s)\Big\rangle_{F^*_{1,2}}ds\nonumber\\
=&&\!\!\!\!\!\!\!\!2\int_0^t\Big\langle \int_Z f(s,X^{g}(s),z)\big((g_n(s,z)-1)-(g(s,z)-1)\big)\nu(dz),X^{g_n}(s)-X^g(s)\Big\rangle_{F^*_{1,2}}ds\nonumber\\
&&\!\!\!\!\!\!\!\!+2\int_0^t\Big\langle \int_Z \big(f(s,X^{g_n}(s),z)-f(s,X^{g}(s),z)\big)\big(g_n(s,z)-1\big)\nu(dz),X^{g_n}(s)-X^g(s)\Big\rangle_{F^*_{1,2}}ds\nonumber\\
:=&&\!\!\!\!\!\!\!\!Q_{n,1}(t)+Q_{n,2}(t).
\end{eqnarray}
For $Q_{n,2}(t)$, notice that from \textbf{(H2)(i)}, we know
\begin{eqnarray}\label{5.04}
|Q_{n,2}(t)|\leq 2\int_0^t\int_Z l_1(s,z)|g_n(s,z)-1|\cdot\|X^{g_n}(s)-X^g(s)\|^2_{F^*_{1,2}}\nu(dz)ds.
\end{eqnarray}
Denoting
\begin{eqnarray}\label{5.040}
h_n(s):=\int_Z l_1(s,z)|g_n(s,z)-1|\nu(dz),
\end{eqnarray}
from \eqref{h0}, we know that
\begin{eqnarray}\label{eq zhai 4}
\sup_{n\geq1}\int_0^Th_n(s)ds\leq C_{l_1,N}.
\end{eqnarray}


substituting \eqref{5.01}-\eqref{5.040} into \eqref{5.00}, we get
\begin{eqnarray*}
&&\!\!\!\!\!\!\!\!\|X^{g_n}(t)-X^g(t)\|^2_{F^*_{1,2}}\nonumber\\
\leq&&\!\!\!\!\!\!\!\!\frac{1}{2\tilde{\alpha}}\int_0^t\|X^{g_n}(s)-X^g(s)\|^2_{F^*_{1,2}}ds+Q_{n,1}(t)+2\int_0^th_n(s)\cdot\|X^{g_n}(s)-X^g(s)\|^2_{F^*_{1,2}}ds.
\end{eqnarray*}
By Gronwall's inequality and \eqref{eq zhai 4}, we obtain
\begin{eqnarray}\label{convergence1}
\sup_{t\in[0,T]}\|X^{g_n}(t)-X^g(t)\|^2_{F^*_{1,2}}
&\leq&\sup_{t\in[0,T]}|Q_{n,1}(t)|\cdot e^{\frac{T}{2\tilde{\alpha}}+2\int_0^Th_n(s)ds}\nonumber\\
&\leq&\sup_{t\in[0,T]}|Q_{n,1}(t)|\cdot e^{\frac{T}{2\tilde{\alpha}}+2C_{l_1,N}}\nonumber\\
&:=&C_{l_1,T,N}\cdot\sup_{t\in[0,T]}|Q_{n,1}(t)|.
\end{eqnarray}

Now, let us estimate $|Q_{n,1}(t)|$. By \eqref{eq 4} and (\ref{eq H infinit H}),
we know that for any $\varepsilon>0$, there exists a compact subset $K_\varepsilon\subset Z$ such that the following  holds:
\begin{eqnarray}\label{eq zhai 1}
\sup_{i=1,2}\sup_{h\in S^N}\int_0^T\int_{K_\varepsilon^c}l_i(s,z)|h(s,z)-1|\nu(dz)ds\leq\varepsilon.
\end{eqnarray}
 Then we have
\begin{eqnarray}\label{Qn1}
Q_{n,1}(t)=&&\!\!\!\!\!\!\!\!2\int_0^t\int_{K_\varepsilon}\Big\langle f(s,X^g(s),z)\big(g_n(s,z)-g(s,z)\big), X^{g_n}(s)-X^g(s)\Big\rangle_{F^*_{1,2}}\nu(dz)ds\nonumber\\
&&\!\!\!\!\!\!\!\!+2\int_0^t\int_{K_\varepsilon^c}\Big\langle f(s,X^g(s),z)\big(g_n(s,z)-g(s,z)\big), X^{g_n}(s)-X^g(s)\Big\rangle_{F^*_{1,2}}\nu(dz)ds\nonumber\\
:=&&\!\!\!\!\!\!\!\!I_{n,1}(t)+I_{n,2}(t).
\end{eqnarray}
Since \eqref{defi4} holds, we know that, for any $t\in[0,T]$,
\begin{eqnarray}\label{In2}
|I_{n,2}(t)|&&\!\!\!\!\!\!\!\!\leq2\int_0^T\int_{K_\varepsilon^c}l_2(s,z)\big(\|X^g(s)\|_{F^*_{1,2}}+1\big)\cdot\nonumber\\
&&\!\!\!\!\!\!\!\!~~~~~~~~~~~~~~~~~~~~~~~~~~~~~~~~~~~~~|g_n(s,z)-g(s,z)|\cdot\|X^{g_n}(s)-X^g(s)\|_{F^*_{1,2}}\nu(dz)ds\nonumber\\
&&\!\!\!\!\!\!\!\!\leq2\Big[\sup_{s\in[0,T]}\big(\|X^g(s)\|_{F^*_{1,2}}+1\big)\big(\|X^{g_n}(s)\|_{F^*_{1,2}}+\|X^g(s)\|_{F^*_{1,2}}\big)\Big]\cdot\nonumber\\
&&\!\!\!\!\!\!\!\!~~~~~\Big(\int_0^T\int_{K_\varepsilon^c}l_2(s,z)|g_n(s,z)-1|\nu(dz)ds+\int_0^T\int_{K_\varepsilon^c}l_2(s,z)|g(s,z)-1|\nu(dz)ds\Big)\nonumber\\
&&\!\!\!\!\!\!\!\!\leq C_{N,T}(|x|_2^2+1)\cdot \varepsilon.
\end{eqnarray}

To estimate $I_{n,1}(t)$, define
\begin{eqnarray*}
A_{2,J}=\{(s,z)\in[0,T]\times Z:l_2(s,z)\geq J\}.
\end{eqnarray*}
For a subset $A\subset[0,T]\times Z$, in the following,
let $A^c$ denote the complement of $A$.

Denote
\begin{eqnarray*}
I_{n,1,J}(t)=2\int_0^t\int_{K_\varepsilon}\big\langle f(s,X^g(s),z)\big(g_n(s,z)-g(s,z)\big),X^{g_n(s)}-X^g(s)\big\rangle_{F^*_{1,2}} 1_{A_{2,J}}(s,z)\nu(dz)ds,
\end{eqnarray*}
\begin{eqnarray*}
I_{n,1,J^c}(t)=2\int_0^t\int_{K_\varepsilon}\big\langle f(s,X^g(s),z)\big(g_n(s,z)-g(s,z)\big),X^{g_n(s)}-X^g(s)\big\rangle_{F^*_{1,2}} 1_{A_{2,J}^c}(s,z)\nu(dz)ds.
\end{eqnarray*}
Then,
$$I_{n,1}(t)=I_{n,1,J}(t)+I_{n,1,J^c}(t).$$

Let us estimate $I_{n,1,J}(t)$ and $I_{n,1,J^c}(t)$ separately. Notice that from \eqref{defi4}, for any $t\in[0,T]$,
\begin{eqnarray}\label{ld4}
&&\!\!\!\!\!\!\!\!|I_{n,1,J}(t)|\nonumber\\
\leq&&\!\!\!\!\!\!\!\!2\int_0^T\int_{K_\varepsilon}l_2(s,z)\big(\|X^g(s)\|_{F^*_{1,2}}+1\big)\big(g_n(s,z)+g(s,z)\big)\cdot\nonumber\\
&&\!\!\!\!\!\!\!\!~~~~~~~~~~~~~~~~~~~~~~~~~~~~~~~~~~~\big(\|X^{g_n}(s)\|_{F^*_{1,2}}+\|X^g(s)\|_{F^*_{1,2}}\big)1_{A_{2,J}}(s,z)\nu(dz)ds\nonumber\\
\leq&&\!\!\!\!\!\!\!\! 2\sup_{s\in[0,T]}\Big[\big(\|X^g(s)\|_{F^*_{1,2}}+1\big)\big(\|X^{g_n}(s)\|_{F^*_{1,2}}+\|X^g(s)\|_{F^*_{1,2}}\big)\Big]\cdot\nonumber\\
&&\!\!\!\!\!\!\!\!~~~~~~~~~~~~~~~~~~~~~~~~~~~~~~~~~~\int_0^T\int_{K_\varepsilon}l_2(s,z)\big(g_n(s,z)+g(s,z)\big)1_{A_{2,J}}(s,z)\nu(dz)ds\nonumber\\
\leq &&\!\!\!\!\!\!\!\!C_{N,T}(|x|_2^2+1)\cdot\sup_{h\in S^N}\int_0^T\int_{K_\varepsilon}l_2(s,z)h(s,z)1_{A_{2,J}}(s,z)\nu(dz)ds.
\end{eqnarray}
By \eqref{eq 2} and (\ref{eq H infinit H}), we know that for $\varepsilon>0$, there exists $J_\varepsilon>0$, such that
\begin{eqnarray*}
\sup_{h\in S^N}\int_0^T\int_{K_\varepsilon}l_2(s,z)h(s,z)1_{\{l_2(s,z)\geq J_\varepsilon\}}\nu(dz)ds\leq \frac{\varepsilon}{C_{N,T}(|x|_2^2+1)}.
\end{eqnarray*}
If we choose $J$ in \eqref{ld4} to be $J_\epsilon$, then \eqref{ld4} yields
\begin{eqnarray}\label{In11}
\sup_{t\in[0,T]}|I_{n,1,J_\varepsilon}(t)|\leq\varepsilon.
\end{eqnarray}
Substituting \eqref{Qn1}-\eqref{In11} into \eqref{convergence1}, we get
\begin{eqnarray}\label{ld4.1}
&&\!\!\!\!\!\!\!\!\sup_{t\in[0,T]}\|X^{g_n}(t)-X^g(t)\|^2_{F^*_{1,2}}\nonumber\\
\leq&&\!\!\!\!\!\!\!\!C_{l_1,T,N}\cdot\sup_{t\in[0,T]}\Big(|I_{n,1,J_\varepsilon}(t)|+|I_{n,1,J_\varepsilon^c}(t)|+|I_{n,2}(t)|\Big)\nonumber\\
\leq&&\!\!\!\!\!\!\!\!C_{l_1,T,N}\cdot\sup_{t\in[0,T]}\Big(\varepsilon+|I_{n,1,J_\varepsilon^c}(t)|+C_{N,T}(|x|_2^2+1)\cdot\varepsilon\Big).
\end{eqnarray}

To estimate $|I_{n,1,J_\varepsilon^c}(t)|$, denote
\begin{eqnarray*}
U^n(s)=X^{g_n}(s)-X^g(s),~~~U^n(\bar{s}_m)=X^{g_n}(\bar{s}_m)-X^g(\bar{s}_m),
\end{eqnarray*}
where
\begin{eqnarray*}
\bar{s}_m=t_{k+1}\equiv(k+1)T\cdot2^{-m},~~\text{for}~~s\in[kT2^{-m}, (k+1)T2^{-m}).
\end{eqnarray*}
Then
\begin{eqnarray}\label{In1}
\sup_{t\in[0,T]}|I_{n,1,J^c_\varepsilon}(t)|\leq\sum_{i=1}^4\tilde{I_i},
\end{eqnarray}
where
\begin{eqnarray*}
&&\!\!\!\!\!\!\!\!\tilde{I_1}=\!\!\sup_{t\in[0,T]}\Big|\!\!\int_0^t\!\!\int_{K_\varepsilon}\!\!\big\langle\!f(s,X^g(s),z)\big(g_n(s,z)-g(s,z)\big),U^n(s)-U^n(\bar{s}_m)\big\rangle_{F^*_{1,2}}1_{A^c_{2,J_\varepsilon}}(s,z)\nu(dz)ds\Big|,\\
&&\!\!\!\!\!\!\!\!\tilde{I_2}=\!\!\sup_{t\in[0,T]}\Big|\!\!\int_0^t\!\!\int_{K_\varepsilon}\!\!\big\langle\!\big(f(s,X^g(s),z)-f(s,X^g(\bar{s}_m),z)\big)\big(g_n(s,z)-g(s,z)\big),\nonumber\\
&&\!\!\!\!\!\!\!\!~~~~~~~~~~~~~~~~~~~~~~~~~~~~~~~~~~~~~~~~~~~~~~~~~~~~~~~~~~~~~~~~~~~~~~~~~~~~U^n(\bar{s}_m)\big\rangle_{F^*_{1,2}}1_{A^c_{2,J_\varepsilon}}(s,z)\nu(dz)ds\Big|,\\
&&\!\!\!\!\!\!\!\!\tilde{I_3}=\!\!\sup_{1\leq k\leq2^m}\!\sup_{t_{k-1}\leq t\leq t_k}\!\Big|\!\!\int_{t_{k-1}}^t\!\!\int_{K_\varepsilon}\!\!\big\langle\!\!f(s,X^g(\bar{s}_m),z)\big(g_n(s,z)-g(s,z)\big),U^n(\bar{s}_m)\big\rangle_{F^*_{1,2}}\!1_{A^c_{2,J_\varepsilon}}(s,z)\nu(dz)ds\Big|,\\
&&\!\!\!\!\!\!\!\!\tilde{I_4}=\!\!\sum_{k=1}^{2^m}\Big|\!\!\int_{t_{k-1}}^{t_k}\!\!\int_{K_\varepsilon}\!\!\big\langle\!f(s,X^g(\bar{s}_m),z)\big(g_n(s,z)-g(s,z)\big),U^n(\bar{s}_m)\big\rangle_{F^*_{1,2}}1_{A^c_{2,J_\varepsilon}}(s,z)\nu(dz)ds\Big|.
\end{eqnarray*}
Notice that $\tilde{I_i}, i=1,...,4$ are all dependent on $n,m,\epsilon,$ so to shorten the notation, we omit these parameters.

Now, let us estimate $\tilde{I_i}, i=1,2,3,4$.
\begin{eqnarray}\label{ld5}
\tilde{I_1}&&\!\!\!\!\!\!\!\!\leq\int_0^T\!\!\int_{K_\varepsilon}\!\!l_2(s,z)1_{A^c_{2,J_\varepsilon}}(s,z)\big(\|X^g(s)\|_{F^*_{1,2}}+1\big)\|U^n(s)-U^n(\bar{s}_m)\|_{F^*_{1,2}}\big(g_n(s,z)+g(s,z)\big)\nu(dz)ds\nonumber\\
&&\!\!\!\!\!\!\!\!\leq \Big(\sup_{s\in[0,T]}\|X^g(s)\|_{F^*_{1,2}}+1\Big)J_\varepsilon\int_0^T\!\!\int_{K_\varepsilon}\!\!\|U^n(s)-U^n(\bar{s}_m)\|_{F^*_{1,2}}\big(g_n(s,z)+g(s,z)\big)\nu(dz)ds.
\end{eqnarray}
Recall that from \cite[Remark 3.3]{BCD}, for any $a, b\in (0,\infty)$ and $\sigma\in [1,\infty)$,
\begin{eqnarray}\label{ld6}
ab\leq e^{\sigma a}+\frac{1}{\sigma}(b\log b-b+1)=e^{\sigma a}+\frac{1}{\sigma}\Phi(b),
\end{eqnarray}
where $\Phi$ is defined as in \eqref{l}.
We choose $a=1$ and $b=g_n(s,z)$ or $g(s,z)$,
so by \eqref{sn} and \eqref{defi4},
\eqref{ld5} can be estimated by
\begin{eqnarray}\label{ld7}
\tilde{I_1}&&\!\!\!\!\!\!\!\!\leq \Big(\sup_{s\in[0,T]}\|X^g(s)\|_{F^*_{1,2}}+1\Big)J_\varepsilon\cdot\!\!\int_0^T\!\!\int_{K_\varepsilon}\!\!\|U^n(s)-U^n(\bar{s}_m)\|_{F^*_{1,2}}\cdot\nonumber\\
&&\!\!\!\!\!\!\!\!~~~~~~~~~~~~~~~~~~~~~~~~~~~~~~~~~~~~~~~~~~~~~~~~~~~~~\Big(2e^{\sigma}+\frac{1}{\sigma}\Phi\big(g_n(s,z)\big)+\frac{1}{\sigma}\Phi\big(g(s,z)\big)\Big)\nu(dz)ds\nonumber\\
&&\!\!\!\!\!\!\!\!\leq \Big(\sup_{s\in[0,T]}\|X^g(s)\|_{F^*_{1,2}}+1\Big)J_\varepsilon \cdot2e^{\sigma}\int_0^T\int_{K_\varepsilon}\|U^n(s)-U^n(\bar{s}_m)\|_{F^*_{1,2}}\nu(dz)ds\nonumber\\
&&\!\!\!\!\!\!\!\!~~+\frac{4}{\sigma}\Big(\sup_{s\in[0,T]}\|X^g(s)\|_{F^*_{1,2}}+1\Big)J_\varepsilon\cdot\sup_{s\in[0,T]}\big(\|X^{g_n}(s)\|_{F^*_{1,2}}+\|X^{g}(s)\|_{F^*_{1,2}}\big)\cdot\nonumber\\
&&\!\!\!\!\!\!\!\!~~~~~~~~~~~~~~~~~~~~~~~~~~~~~~~~~~~~~~~~~~~~~~~~~~~~~~~~~~~~~~~~~~~~\sup_{h\in S^N}\int_0^T\int_{K_\varepsilon}\Phi\big(h(s,z)\big)\nu(dz)ds\nonumber\\
&&\!\!\!\!\!\!\!\!\leq \Big(\sup_{s\in[0,T]}\|X^g(s)\|_{F^*_{1,2}}+1\Big)J_\varepsilon\cdot2e^{\sigma}\nu(K_\varepsilon)\cdot\sqrt{T}\cdot\nonumber\\
&&\!\!\!\!\!\!\!\!~~~~~~~~~~~~~~~~\Big[\big(\int_0^T\|X^{g_n}(s)-X^{g_n}(\bar{s}_m)\|_{F^*_{1,2}}^2ds\big)^{\frac{1}{2}}+\big(\int_0^T\|X^{g}(s)-X^{g}(\bar{s}_m)\|_{F^*_{1,2}}^2ds\big)^{\frac{1}{2}}\Big]\nonumber\\
&&\!\!\!\!\!\!\!\!~~+8\Big(\sup_{s\in[0,T]}\sup_{\hbar\in S^N}\|X^{\hbar}(s)\|_{F^*_{1,2}}+1\Big)^2J_\varepsilon \cdot \frac{N}{\sigma}\nonumber\\
&&\!\!\!\!\!\!\!\!\leq C_{N,T}(|x|_2^2+1)J_\varepsilon\cdot e^{\sigma}\nu(K_\varepsilon)\cdot\nonumber\\
&&\!\!\!\!\!\!\!\!~~~~~~~~~~~~~~~~\Big[\big(\int_0^T\|X^{g_n}(s)-X^{g_n}(\bar{s}_m)\|_{F^*_{1,2}}^2ds\big)^{\frac{1}{2}}+\big(\int_0^T\|X^{g}(s)-X^{g}(\bar{s}_m)\|_{F^*_{1,2}}^2ds\big)^{\frac{1}{2}}\Big]\nonumber\\
&&\!\!\!\!\!\!\!\!~~+C_{N,T}(|x|_2^2+1)J_\varepsilon \cdot \frac{1}{\sigma}
.
\end{eqnarray}

To estimate $\int_0^T\|X^{g_n}(s)-X^{g_n}(\bar{s}_m)\|_{F^*_{1,2}}^2ds$, notice that from \eqref{ld3} we know
\begin{eqnarray}\label{5.10}
X^{g_n}(\bar{s}_m)-X^{g_n}(s)-\int_s^{\bar{s}_m}\!\!L\Psi(X^{g_n}(t))dt=\int_s^{\bar{s}_m}\!\!\int_Z\!\!f(t,X^{g_n}(t),z)\big(g_n(t,z)-1\big)\nu(dz)dt.
\end{eqnarray}
Applying the chain rule to $\|X^{g_n}(\bar{s}_m)-X^{g_n}(s)\|^2_{F^*_{1,2}}$, we get
\begin{eqnarray}\label{5.11}
&&\!\!\!\!\!\!\!\!\|X^{g_n}(\bar{s}_m)-X^{g_n}(s)\|^2_{F^*_{1,2}}\nonumber\\
&&\!\!\!\!\!\!\!\!+2\int_s^{\bar{s}_m}~ _{(L^2(\mu))^*}\big\langle(1-L)\big(\Psi(X^{g_n}(t))\big),X^{g_n}(t)-X^{g_n}(s)\big\rangle_{L^2(\mu)}dt\nonumber\\
=&&\!\!\!\!\!\!\!\!2\int_s^{\bar{s}_m}\big\langle\Psi(X^{g_n}(t)),X^{g_n}(t)-X^{g_n}(s)\big\rangle_{F^*_{1,2}}dt\nonumber\\
&&\!\!\!\!\!\!\!\!+2\int_s^{\bar{s}_m}\big\langle \int_Zf(t,X^{g_n}(t),z)\big(g_n(t,z)-1\big)\nu(dz),X^{g_n}(t)-X^{g_n}(s)\big\rangle_{F^*_{1,2}}dt.
\end{eqnarray}
Integrating \eqref{5.11} over $[0,T]$ with respect to $s$, we obtain
\begin{eqnarray}\label{5.12}
&&\!\!\!\!\!\!\!\!\int_0^T\|X^{g_n}(\bar{s}_m)-X^{g_n}(s)\|^2_{F^*_{1,2}}ds\nonumber\\
&&\!\!\!\!\!\!\!\!+2\int_0^T\int_s^{\bar{s}_m}~ _{(L^2(\mu))^*}\big\langle(1-L)\big(\Psi(X^{g_n}(t))\big),X^{g_n}(t)-X^{g_n}(s)\big\rangle_{L^2(\mu)}dtds\nonumber\\
=&&\!\!\!\!\!\!\!\!2\int_0^T\int_s^{\bar{s}_m}\big\langle\Psi(X^{g_n}(t)),X^{g_n}(t)-X^{g_n}(s)\big\rangle_{F^*_{1,2}}dtds\nonumber\\
&&\!\!\!\!\!\!\!\!+2\int_0^T\int_s^{\bar{s}_m}\big\langle \int_Zf(t,X^{g_n}(t),z)\big(g_n(t,z)-1\big)\nu(dz),X^{g_n}(t)-X^{g_n}(s)\big\rangle_{F^*_{1,2}}dtds.
\end{eqnarray}
Also, we can rewrite \eqref{5.12} as
\begin{eqnarray*}
&&\!\!\!\!\!\!\!\!\int_0^T\|X^{g_n}(\bar{s}_m)-X^{g_n}(s)\|^2_{F^*_{1,2}}ds\nonumber\\
&&\!\!\!\!\!\!\!\!+2\int_0^T\int_s^{\bar{s}_m}~ _{(L^2(\mu))^*}\big\langle(1-L)\big(\Psi(X^{g_n}(t))-\Psi(X^{g_n}(s))\big),X^{g_n}(t)-X^{g_n}(s)\big\rangle_{L^2(\mu)}dtds\nonumber\\
&&\!\!\!\!\!\!\!\!+2\int_0^T\int_s^{\bar{s}_m}~ _{(L^2(\mu))^*}\big\langle(1-L)\big(\Psi(X^{g_n}(s))\big),X^{g_n}(t)-X^{g_n}(s)\big\rangle_{L^2(\mu)}dtds\nonumber\\
=&&\!\!\!\!\!\!\!\!2\int_0^T\int_s^{\bar{s}_m}\big\langle\Psi(X^{g_n}(t))-\Psi(X^{g_n}(s)),X^{g_n}(t)-X^{g_n}(s)\big\rangle_{F^*_{1,2}}dtds\nonumber\\
&&\!\!\!\!\!\!\!\!+2\int_0^T\int_s^{\bar{s}_m}\big\langle\Psi(X^{g_n}(s)),X^{g_n}(t)-X^{g_n}(s)\big\rangle_{F^*_{1,2}}dtds\nonumber\\
&&\!\!\!\!\!\!\!\!+2\int_0^T\int_s^{\bar{s}_m}\big\langle \int_Zf(t,X^{g_n}(t),z)\big(g_n(t,z)-1\big)\nu(dz),X^{g_n}(t)-X^{g_n}(s)\big\rangle_{F^*_{1,2}}dtds.
\end{eqnarray*}
By \eqref{triple2}, \eqref{defi4}, \eqref{psii}, $L^2(\mu)\subset F^*_{1,2}$ continuously and densely, Young's inequality, $\Psi(0)=0$ (see \textbf{(H1)}), and \textbf{(H2)(ii)},
we obtain
\begin{eqnarray}\label{5.13}
&&\!\!\!\!\!\!\!\!\int_0^T\|X^{g_n}(\bar{s}_m)-X^{g_n}(s)\|^2_{F^*_{1,2}}ds\nonumber\\
\leq&&\!\!\!\!\!\!\!\!2\int_0^T\int_s^{\bar{s}_m}|\Psi(X^{g_n}(s))|_2\cdot|X^{g_n}(t)-X^{g_n}(s)|_2dtds\nonumber\\
&&\!\!\!\!\!\!\!\!+\frac{1}{2\tilde{\alpha}}\int_0^T\int_s^{\bar{s}_m}\|X^{g_n}(t)-X^{g_n}(s)\|^2_{F^*_{1,2}}dtds\nonumber\\
&&\!\!\!\!\!\!\!\!+2\int_0^T\int_s^{\bar{s}_m}|\Psi(X^{g_n}(s))|_2\cdot \|X^{g_n}(t)-X^{g_n}(s)\|_{F^*_{1,2}}dtds\nonumber\\
&&\!\!\!\!\!\!\!\!+2\int_0^T\int_s^{\bar{s}_m}\int_Zl_2(t,z)\cdot\big(\|X^{g_n}(t)\|_{F^*_{1,2}}+1\big)\cdot|g_n(t,z)-1|\cdot\|X^{g_n}(t)-X^{g_n}(s)\|_{F^*_{1,2}}\nu(dz)dtds\nonumber\\
\leq&&\!\!\!\!\!\!\!\!2\int_0^T\int_s^{\bar{s}_m}Lip\Psi\cdot|X^{g_n}(s)|_2\cdot|X^{g_n}(t)-X^{g_n}(s)|_2dtds\nonumber\\
&&\!\!\!\!\!\!\!\!+\frac{1}{2\tilde{\alpha}}\int_0^T\int_s^{\bar{s}_m}|X^{g_n}(t)-X^{g_n}(s)|^2_2dtds\nonumber\\
&&\!\!\!\!\!\!\!\!+2\int_0^T\int_s^{\bar{s}_m}Lip\Psi\cdot|X^{g_n}(s)|_2\cdot |X^{g_n}(t)-X^{g_n}(s)|_2dtds\nonumber\\
&&\!\!\!\!\!\!\!\!+2\int_0^T\int_s^{\bar{s}_m}\big(|X^{g_n}(t)|_2+1\big)\cdot|X^{g_n}(t)-X^{g_n}(s)|_2\cdot\int_Zl_2(t,z)\cdot|g_n(t,z)-1|\nu(dz)dtds\nonumber\\
\leq&&\!\!\!\!\!\!\!\!C_{N,T}(|x|_2^2+1)\cdot\frac{1}{2^m}+C_{N,T}(|x|_2^2+1)\cdot\int_0^T\int_s^{\bar{s}_m}\int_Zl_2(t,z)\cdot|g_n(t,z)-1|\nu(dz)dtds\nonumber\\
\leq&&\!\!\!\!\!\!\!\!C_{N,T}(|x|_2^2+1)\cdot\frac{1}{2^m}\nonumber\\
&&\!\!\!\!\!\!\!\!+C_{N,T}(|x|_2^2+1)\cdot\sup_{\hbar\in S^N}\sup_{s\in[0,T]}\int_s^{\bar{s}_m}\int_Zl_2(t,z)\cdot|\hbar(t,z)-1|\nu(dz)dt.
\end{eqnarray}
Then, from \eqref{eq 3} and (\ref{eq H infinit H}) we know that
\begin{eqnarray}\label{5.14}
\lim_{\delta\rightarrow0}\sup_{\hbar\in S^N}\sup_{|l-s|\leq\delta}\int_s^l\int_Zl_i(t,z)|\hbar(t,z)-1|\nu(dz)dt=0,~~i=1,2,
\end{eqnarray}
so
\begin{eqnarray}\label{5.15}
\lim_{m\rightarrow\infty}\sup_{n\in \Bbb{N}}\int_0^T\|X^{g_n}(\bar{s}_m)-X^{g_n}(s)\|^2_{F^*_{1,2}}=0,
\end{eqnarray}
and similarly,
\begin{eqnarray}\label{5.17}
&&\!\!\!\!\!\!\!\!\lim_{m\rightarrow\infty}\int_0^T\|X^{g}(s)-X^{g}(\bar{s}_m)\|^2_{F^*_{1,2}}ds=0.
\end{eqnarray}
Substituting \eqref{5.15} and \eqref{5.17} into \eqref{ld7}, we know that
\begin{eqnarray}\label{I11}
\overline{\lim_{m\rightarrow\infty}}\sup_{n\in\Bbb{N}}\tilde{I_1}\leq C_{N,T}(|x|_2^2+1)J_\varepsilon\cdot\frac{1}{\sigma},
\end{eqnarray}
and since $\sigma$ is arbitrary in $[1,\infty)$, we get
\begin{eqnarray}\label{I1}
\lim_{m\rightarrow\infty}\sup_{n\in\Bbb{N}}\tilde{I_1}=0.
\end{eqnarray}

To estimate $\tilde{I_2}$, denote
$$A_{1,J}=\{(s,z)\in [0,T]\times Z:l_1(s,z)\geq J\},$$
so by \textbf{(H2)(i)}, \eqref{defi4}, and \eqref{ld6},
we have
\begin{eqnarray}\label{ld8}
\tilde{I_2}\leq&&\!\!\!\!\!\!\!\!\int_0^T\!\!\int_{K_\varepsilon}\!\!l_1(s,z)\|X^g(s)-X^g(\bar{s}_m)\|_{F^*_{1,2}}\|U^n(\bar{s}_m)\|_{F^*_{1,2}}|g_n(s,z)-g(s,z)|\nu(dz)ds\nonumber\\
=&&\!\!\!\!\!\!\!\!\int_0^T\int_{K_\varepsilon}l_1(s,z)\|X^g(s)-X^g(\bar{s}_m)\|_{F^*_{1,2}}\|U^n(\bar{s}_m)\|_{F^*_{1,2}}|g_n(s,z)-g(s,z)|1_{A_{1,J}}(s,z)\nu(dz)ds\nonumber\\
&&\!\!\!\!\!\!\!\!+\int_0^T\int_{K_\varepsilon}l_1(s,z)\|X^g(s)-X^g(\bar{s}_m)\|_{F^*_{1,2}}\|U^n(\bar{s}_m)\|_{F^*_{1,2}}|g_n(s,z)-g(s,z)|1_{A_{1,J}^c}(s,z)\nu(dz)ds\nonumber\\
\leq&&\!\!\!\!\!\!\!\! \sup_{\hbar\in S^N}\sup_{s\in[0,T]}4\|X^{\hbar}(s)\|^2_{F^*_{1,2}}\cdot\int_0^T\int_{K_\varepsilon}l_1(s,z)\big(|g_n(s,z)|+|g(s,z)|\big)1_{A_{1,J}}(s,z)\nu(dz)ds\nonumber\\
&&\!\!\!\!\!\!\!\!+J\sup_{\hbar\in S^N}\sup_{s\in[0,T]}2\|X^{\hbar}(s)\|_{F^*_{1,2}}\cdot\int_0^T\int_{K_\varepsilon}\|X^g(s)-X^g(\bar{s}_m)\|_{F^*_{1,2}}\big(|g_n(s,z)|+|g(s,z)|\big)\nu(dz)ds\nonumber\\
\leq&&\!\!\!\!\!\!\!\!C_{N,T}(|x|_2^2+1)\sup_{h\in S^N}\int_0^T\int_{K_\varepsilon}l_1(s,z)h(s,z)1_{A_{1,J}}(s,z)\nu(dz)ds\nonumber\\
&&\!\!\!\!\!\!\!\!+JC_{N,T}(|x|_2^2+1)\int_0^T\int_{K_\varepsilon}\|X^g(s)-X^g(\bar{s}_m)\|_{F^*_{1,2}}\cdot\nonumber\\
&&\!\!\!\!\!\!\!\!~~~~~~~~~~~~~~~~~~~~~~~~~~~~~~~~~~~~~~~~~~~~~~~~~~~~~~~~~~\big(2e^{\sigma}+\frac{1}{\sigma}\Phi(g_n(s,z))+\frac{1}{\sigma}\Phi(g(s,z))\big)\nu(dz)ds\nonumber\\
\leq&&\!\!\!\!\!\!\!\!C_{N,T}(|x|_2^2+1)\sup_{h\in S^N}\int_0^T\int_{K_\varepsilon}l_1(s,z)h(s,z)1_{A_{1,J}}(s,z)\nu(dz)ds\nonumber\\
&&\!\!\!\!\!\!\!\!+JC_{N,T}(|x|_2^2+1)e^{\sigma}\nu(K_\varepsilon)\int_0^T\|X^g(s)-X^g(\bar{s}_m)\|_{F^*_{1,2}}ds\nonumber\\
&&\!\!\!\!\!\!\!\!+JC_{N,T}(|x|_2^2+1)\int_0^T\int_{K_\varepsilon}\frac{1}{\sigma}\Phi(g_n(s,z))+\frac{1}{\sigma}\Phi(g(s,z))\nu(dz)ds.
\end{eqnarray}
For the first term in the right hand-side of \eqref{ld8}, from \eqref{eq 4} we know that for the fixed $K_\varepsilon$ and any $\eta>0$, there exists $J_\eta>0$ such that
\begin{eqnarray}\label{5.18}
\sup_{h\in S^N}\int_0^T\int_{K_\varepsilon}l_1(s,z)h(s,z)1_{A_{1,J_\eta}}(s,z)\nu(dz)ds\leq\eta.
\end{eqnarray}
Fix the above  $J_\eta$.
The second term in the right hand-side of \eqref{ld8} is dominated by
\begin{eqnarray}\label{5.19}
\leq&&\!\!\!\!\!\!\!\!J_\eta C_{N,T}(|x|_2^2+1) e^\sigma\nu(K_\varepsilon)\Big(\int_0^T\|X^g(s)-X^g(\bar{s}_m)\|^2_{F^*_{1,2}}ds\Big)^{\frac{1}{2}}.
\end{eqnarray}
From \eqref{sn}, the third term in the right hand-side of \eqref{ld8} is  dominated by
\begin{eqnarray}\label{5.20}
\leq&&\!\!\!\!\!\!\!\!C_{N,T}(|x|^2_2+1)J_\eta\cdot\frac{1}{\sigma}.
\end{eqnarray}
Taking \eqref{5.18}-\eqref{5.20} into \eqref{ld8}, and since $\sigma$ is arbitrary in $[1,\infty)$, by \eqref{5.17},
\begin{eqnarray*}
\overline{\lim_{m\rightarrow\infty}}\sup_{n\in\Bbb{N}}\tilde{I_2}\leq C_{N,T}(|x|_2^2+1)\eta.
\end{eqnarray*}
Since $\eta$ is arbitrary in $(0,\infty)$, we obtain
\begin{eqnarray}\label{I2}
\lim_{m\rightarrow\infty}\sup_{n\in\Bbb{N}}\tilde{I_2}=0.
\end{eqnarray}

By \textbf{(H2)(ii)} and \eqref{defi4},
\begin{eqnarray}\label{ld9}
\tilde{I_3}\leq&&\!\!\!\!\!\!\!\!\sup_{1\leq k\leq 2^m}\sup_{t_{k-1}\leq t\leq t_k}\int_{t_{k-1}}^t\int_{K_\varepsilon}l_2(s,z)\big(\|X^g(\bar{s}_m)\|_{F^*_{1,2}}+1\big)\cdot\|U^n(\bar{s}_m)\|_{F^*_{1,2}}\cdot\nonumber\\
&&\!\!\!\!\!\!\!\!~~~~~~~~~~~~~~~~~~~~~~~~~~~~~~~~~~~~~~~~~~~~~~~~~\big(|g_n(s,z)-1|+|g(s,z)-1|\big)\nu(dz)ds\nonumber\\
\leq&&\!\!\!\!\!\!\!\!C_{N,T}(|x|_2^2+1)\sup_{h\in S^N}\sup_{1\leq k\leq 2^m}\sup_{t_{k-1}\leq t\leq t_k}\int_{t_{k-1}}^t\int_{K_\varepsilon}l_2(s,z)\cdot|h(s,z)-1|\nu(dz)ds,
\end{eqnarray}
so applying \eqref{5.14}, we get
\begin{eqnarray}\label{I3}
\overline{\lim_{m\rightarrow\infty}}\sup_{n\in \Bbb{N}}\tilde{I_3}=0.
\end{eqnarray}

To estimate $\tilde{I_4}$, for all $(s,z)\in[0,T]\times Z$, we denote
$$f_{m,n}(s,z)=\langle f(s,X^g(\bar{s}_m),z),U^n(\bar{s}_m)\rangle_{F^*_{1,2}}1_{A^c_{2,J_\varepsilon}}(s,z).$$
Since $U^n(\bar{s}_m)=X^{g_n}(\bar{s}_m)-X^g(\bar{s}_m)$, $L^2(\mu)\subset F^*_{1,2}$ continuously and densely, and \eqref{defi4}, we know that
$$\sup_{n\in\Bbb{N}}\sup_{s\in[0,T]}\|U^n(\bar{s}_m)\|_{F^*_{1,2}}\leq \sup_{n\in\Bbb{N}}\sup_{s\in[0,T]}|U^n(\bar{s}_m)|_2\leq2\sup_{\hbar\in S^N}\sup_{t\in[0,T]}|X^\hbar(t)|_2\leq\sqrt{C_{N,T}(|x|_2^2+1)}<\infty.$$
Fix $m\in\mathbb{N}$ and $k=1,2,...,2^m$. For any $s\in[t_{k-1}, t_k)$,  $\{U^n(\bar{s}_m)=U^n(t_k)\}_{n\geq1}$ is weakly compact in $(F^*_{1,2}, \|\cdot\|_{F^*_{1,2}})$, hence there exists a subsequence, still denoted later by $ U^n(\bar{s}_m)$ and $U_k\in F^*_{1,2}$, such that for all $\kappa\in F^*_{1,2}$, we have
$$\lim_{n\rightarrow\infty}\langle \kappa, U^n(\bar{s}_m)\rangle_{F^*_{1,2}}=\langle \kappa, U_k\rangle_{F^*_{1,2}},~~\text{and}~~\|U_k\|_{F^*_{1,2}}\leq \sqrt{C_{N,T}(|x|_2^2+1)}<\infty.$$
Therefore, on $[t_{k-1},t_k)\times K_\varepsilon$,
\begin{eqnarray}\label{eq Zhai 6}
\lim_{n\rightarrow\infty}f_{m,n}(s,z)=\langle f(s,X^g(\bar{s}_m),z), U_k\rangle_{F^*_{1,2}}1_{A^c_{2,J_\varepsilon}}(s,z):=f_m(s,z),~~\nu(dz)ds\text{-}a.s.,
\end{eqnarray}
and
\begin{eqnarray}
|f_{m,n}(s,z)|\!\leq\! l_2(s,z)\big(\|X^g(\bar{s}_m)\|_{F^*_{1,2}}\!\!+1\big)\|U^n(\bar{s}_m)\|_{F^*_{1,2}}1_{A^c_{2,J_\varepsilon}}\!\!(s,z)\!\leq \! J_\varepsilon C_{N,T}(|x|_2^2+1)\!<\!\infty,
\end{eqnarray}
which imply that
\begin{eqnarray}\label{eq Zhai 6-1}
|f_{m}(s,z)|\leq J_\varepsilon C_{N,T}(|x|_2^2+1)<\infty.
\end{eqnarray}

Recall $\nu_T=\lambda_T\otimes\nu$ and $\nu_T^g$ introduced in (\ref{eq Zhai 5}). Similar to \cite[page:269, Lemma 10.9]{BDLDP}, we can assume without loss of generality that $\nu_T([t_{k-1},t_k)\times K_\varepsilon)\neq0$, $\nu_T\big(\partial([t_{k-1},t_k)\times K_\varepsilon)\big)=0$, and
$$m_{n,\varepsilon}:=\int_{t_{k-1}}^{t_k}\int_{K_\varepsilon}g_n(s,z)\nu(dz)ds\neq0,~~\text{and}~~m_\varepsilon:=\int_{t_{k-1}}^{t_k}\int_{K_\varepsilon}g(s,z)\nu(dz)ds\neq0.$$
Define probability measures $\tilde{\nu}_{n,\varepsilon}$, $\tilde{\nu}_\varepsilon$, and $\theta_\varepsilon$ on $([t_{k-1},t_k)\times K_\varepsilon, \mathcal{B}([t_{k-1},t_k))\otimes\mathcal{B}(K_\varepsilon))$ as follows:
\begin{eqnarray*}
\tilde{\nu}_{n,\varepsilon}(\cdot)&&\!\!\!\!\!\!\!\!=\frac{1}{m_{n,\varepsilon}}\nu_T^{g_n}(\cdot\cap([t_{k-1},t_k)\times K_\varepsilon)),\nonumber\\
\tilde{\nu}_\varepsilon(\cdot)&&\!\!\!\!\!\!\!\!=\frac{1}{m_\varepsilon}\nu_T^g(\cdot\cap([t_{k-1},t_k)\times K_\varepsilon)),\nonumber\\
\theta_\varepsilon(\cdot)&&\!\!\!\!\!\!\!\!=\frac{\nu_T(\cdot\cap([t_{k-1},t_k)\times K_\varepsilon))}{\nu_T([t_{k-1},t_k)\times K_\varepsilon))}.
\end{eqnarray*}
Recall the topology of $S^N$; see the discussion around (\ref{eq Zhai 5}).
Since $g_n\rightarrow g$ as $n\rightarrow\infty$ in $S^N$, we have that
\begin{eqnarray}\label{m}
\lim_{n\rightarrow\infty}m_{n,\varepsilon}=m_\varepsilon,
\end{eqnarray}
\begin{eqnarray}\label{eq Zhai 7}
\tilde{\nu}_{n,\varepsilon}~\text{converges~weakly~to}~\tilde{\nu}_\varepsilon~\text{as}~n\rightarrow\infty.
\end{eqnarray}
Also, there exists a constant $\alpha_\varepsilon$ such that the relative entropy function
\begin{eqnarray}\label{eq Zhai 8}
\sup_{n\geq1}R(\tilde{\nu}_{n,\varepsilon}\|\theta_\varepsilon)=&&\!\!\!\!\!\!\!\!\sup_{n\geq1}\int_{t_{k-1}}^{t_k}\int_{K_\varepsilon}\log\Big(\frac{\nu_T([t_{k-1},t_k)\times K_\varepsilon)}{m_{n,\varepsilon}}g_n(s,z)\Big)\frac{1}{m_{n,\varepsilon}}g_n(s,z)\nu(dz)ds\nonumber\\
=&&\!\!\!\!\!\!\!\!\sup_{n\geq1}\Big(\frac{1}{m_{n,\varepsilon}}\int_{t_{k-1}}^{t_k}\int_{K_\varepsilon}\big(\Phi(g_n(s,z))+g_n(s,z)-1\big)\nu(dz)ds\nonumber\\
&&\!\!\!\!\!\!\!\!+\log\frac{\nu_T([t_{k-1},t_k)\times K_\varepsilon)}{m_{n,\varepsilon}}\Big)\nonumber\\
\leq&&\!\!\!\!\!\!\!\!\sup_{n\geq1}\Big(\frac{N}{m_{n,\varepsilon}}+1-\frac{\nu_T([t_{k-1},t_k)\times K_\varepsilon)}{m_{n,\varepsilon}}+\log\frac{\nu_T([t_{k-1},t_k)\times K_\varepsilon)}{m_{n,\varepsilon}}\Big)\nonumber\\
\leq&&\!\!\!\!\!\!\!\!\alpha_\varepsilon<\infty.
\end{eqnarray}

By (\ref{eq Zhai 6})--(\ref{eq Zhai 6-1}), (\ref{eq Zhai 7}), and (\ref{eq Zhai 8}), applying \cite[Lemma 2.8]{BD}, we have
\begin{eqnarray*}
&&\!\!\!\!\!\!\!\!(a)~~\lim_{n\rightarrow\infty}\int_{t_{k-1}}^{t_k}\int_{K_\varepsilon}f_{m,n}(s,z)\tilde{\nu}_{n,\varepsilon}(dzds)=\int_{t_{k-1}}^{t_k}\int_{K_\varepsilon}f_{m}(s,z)\tilde{\nu}_\varepsilon(dzds),\nonumber\\
&&\!\!\!\!\!\!\!\!(b)~~\lim_{n\rightarrow\infty}\int_{t_{k-1}}^{t_k}\int_{K_\varepsilon}f_{m,n}(s,z)\tilde{\nu}_{\varepsilon}(dzds)=\int_{t_{k-1}}^{t_k}\int_{K_\varepsilon}f_m(s,z)\tilde{\nu}_\varepsilon(dzds),
\end{eqnarray*}
i.e.,
\begin{eqnarray*}
(a')~~&&\!\!\!\!\!\!\!\!\lim_{n\rightarrow\infty}\int_{t_{k-1}}^{t_k}\int_{K_\varepsilon}\big\langle f(s,X^g(\bar{s}_m),z), U^n(\bar{s}_m)\big\rangle_{F^*_{1,2}}1_{A^c_{2,J_\varepsilon}}(s,z)\frac{1}{m_{n,\varepsilon}}g_n(s,z)\nu(dz)ds\nonumber\\
=&&\!\!\!\!\!\!\!\!\int_{t_{k-1}}^{t_k}\int_{K_\varepsilon}\big\langle f(s,X^g(\bar{s}_m),z), U_k\big\rangle_{F^*_{1,2}}1_{A^c_{2,J_\varepsilon}}(s,z)\frac{1}{m_\varepsilon}g(s,z)\nu(dz)ds,\nonumber\\
(b')~~&&\!\!\!\!\!\!\!\!\lim_{n\rightarrow\infty}\int_{t_{k-1}}^{t_k}\int_{K_\varepsilon}\big\langle f(s,X^g(\bar{s}_m),z), U^n(\bar{s}_m)\big\rangle_{F^*_{1,2}}1_{A^c_{2,J_\varepsilon}}(s,z)\frac{1}{m_{\varepsilon}}g(s,z)\nu(dz)ds\nonumber\\
=&&\!\!\!\!\!\!\!\!\int_{t_{k-1}}^{t_k}\int_{K_\varepsilon}\big\langle f(s,X^g(\bar{s}_m),z), U_k\big\rangle_{F^*_{1,2}}1_{A^c_{2,J_\varepsilon}}(s,z)\frac{1}{m_\varepsilon}g(s,z)\nu(dz)ds.
\end{eqnarray*}
Therefore, taking \eqref{m} into account, we get
\begin{eqnarray}\label{I4}
\lim_{n\rightarrow\infty}\Big|\int_{t_{k-1}}^{t_k}\!\!\int_{K_\varepsilon}\langle f(s,X^g(\bar{s}_m),z),U^n(\bar{s}_m)\rangle_{F^*_{1,2}}\big(g_n(s,z)-g(s,z)\big)1_{A^c_{2,J_\varepsilon}}(s,z)\nu(dz)ds\Big|=0.
\end{eqnarray}

From \eqref{I1}, \eqref{I2}, and \eqref{I3}, we know that for any $\kappa>0$, there exists $m_\kappa>0$ such that for all $m\geq m_\kappa$
\begin{eqnarray}\label{I5}
\sum_{i=1}^3\sup_{n\in\mathbb{N}}\tilde{I_{i}}\leq \kappa.
\end{eqnarray}
For the fixed $\kappa$ and $m_\kappa$ as above, \eqref{I4} implies that
\begin{eqnarray}\label{I6}
\lim_{n\rightarrow\infty}\tilde{I_4}=0.
\end{eqnarray}
Taking \eqref{I5} and \eqref{I6} into account, from \eqref{In1}, we get
\begin{eqnarray*}
\lim_{n\rightarrow\infty}\sup_{t\in[0,T]}|I_{n,1,J^c_\varepsilon}(t)|\leq\kappa,
\end{eqnarray*}
and since $\kappa$ is arbitrary in $(0,\infty)$,
\begin{eqnarray}\label{In12}
\lim_{n\rightarrow\infty}\sup_{t\in[0,T]}|I_{n,1,J^c_\varepsilon}(t)|=0.
\end{eqnarray}
Taking \eqref{In12} into account, from \eqref{ld4.1}, we know that
\begin{eqnarray*}
\lim_{n\rightarrow\infty}\sup_{t\in[0,T]}\|X^{g_n}(t)-X^g(t)\|^2_{F^*_{1,2}}
\leq C_{l_1,T,N}\cdot\Big(\varepsilon+C_{N,T}(|x|_2^2+1)\varepsilon\Big),
\end{eqnarray*}
and since $\varepsilon$ is arbitrary in $(0,\infty)$, it follows that
\begin{eqnarray}\label{a}
X^{g_n}\rightarrow X^g~~\text{in}~~D([0,T];F^*_{1,2}),
\end{eqnarray}
which indicates (a) in Condition \ref{Condition}.

\end{proof}

\subsection{Proof of (b) in Condition \ref{Condition}}\label{Subsection5.2}

\begin{proof}{\bf Proof of (b) in Condition \ref{Condition}}

Let $N<\infty$, $\{\varphi_\epsilon\}_{\epsilon>0}\subset\widetilde{\mathcal{A}}^N$. From Theorem \ref{Th2}, we know that $Y^{\varphi_\epsilon}:=\mathcal{G}^0(\varphi_\epsilon)$ is the unique solution of the following equation:
\begin{eqnarray}\label{ldp2}
Y^{\varphi_\epsilon}(t)=x+L\int_0^t\Psi(Y^{\varphi_\epsilon}(s))ds+\int_0^t\int_Zf(s,Y^{\varphi_\epsilon}(s),z)\big(\varphi_\epsilon(s,z)-1\big)\nu(dz)ds.
\end{eqnarray}
From \eqref{ldp1} and \eqref{ldp2}, we can get
\begin{eqnarray}\label{LDP1}
X^{\varphi_\epsilon}(t)-Y^{\varphi_\epsilon}(t)=&&\!\!\!\!\!\!\!\!L\int_0^t\Psi(X^{\varphi_\epsilon}(s))-\Psi(Y^{\varphi_\epsilon}(s))ds+\epsilon\int_0^t\int_Zf(s,X^{\varphi_\epsilon}(s-),z)\widetilde{N}^{\epsilon^{-1}\varphi_\epsilon}(dz,ds)\nonumber\\
&&\!\!\!\!\!\!\!\!+\int_0^t\int_Z\big(f(s,X^{\varphi_\epsilon}(s),z)-f(s,Y^{\varphi_\epsilon}(s),z)\big)\big(\varphi_\epsilon(s,z)-1\big)\nu(dz)ds.
\end{eqnarray}
Rewrite \eqref{LDP1} in the following form:
\begin{eqnarray}\label{LDP2}
&&\!\!\!\!\!\!\!\!X^{\varphi_\epsilon}(t)-Y^{\varphi_\epsilon}(t)+(1-L)\int_0^t\Psi(X^{\varphi_\epsilon}(s))-\Psi(Y^{\varphi_\epsilon}(s))ds\nonumber\\
=&&\!\!\!\!\!\!\!\!\int_0^t\Psi(X^{\varphi_\epsilon}(s))-\Psi(Y^{\varphi_\epsilon}(s))ds+\epsilon\int_0^t\int_Zf(s,X^{\varphi_\epsilon}(s-),z)\widetilde{N}^{\epsilon^{-1}\varphi_\epsilon}(dz,ds)\nonumber\\
&&\!\!\!\!\!\!\!\!+\int_0^t\int_Z\big(f(s,X^{\varphi_\epsilon}(s),z)-f(s,Y^{\varphi_\epsilon}(s),z)\big)\big(\varphi_\epsilon(s,z)-1\big)\nu(dz)ds.
\end{eqnarray}

Applying It\^{o}'s formula to $\|X^{\varphi_\epsilon}(t)-Y^{\varphi_\epsilon}(t)\|^2_{F^*_{1,2}}$, we get
\begin{eqnarray}\label{LDP3}
&&\!\!\!\!\!\!\!\!\|X^{\varphi_\epsilon}(t)-Y^{\varphi_\epsilon}(t)\|^2_{F^*_{1,2}}\nonumber\\
&&\!\!\!\!\!\!\!\!+2\int_0^t~_{(L^2(\mu))^*}\big\langle (1-L)\big(\Psi(X^{\varphi_\epsilon}(s))-\Psi(Y^{\varphi_\epsilon}(s))\big),X^{\varphi_\epsilon}(s)-Y^{\varphi_\epsilon}(s)\big\rangle_{L^2(\mu)}ds\nonumber\\
=&&\!\!\!\!\!\!\!\!2\int_0^t\big\langle\Psi(X^{\varphi_\epsilon}(s))-\Psi(Y^{\varphi_\epsilon}(s)),X^{\varphi_\epsilon}(s)-Y^{\varphi_\epsilon}(s)\big\rangle_{F^*_{1,2}}ds\nonumber\\
&&\!\!\!\!\!\!\!\!+2\epsilon\int_0^t\int_Z\big\langle f(s,X^{\varphi_\epsilon}(s-),z),X^{\varphi_\epsilon}(s-)-Y^{\varphi_\epsilon}(s-)\big\rangle_{F^*_{1,2}}\widetilde{N}^{\epsilon^{-1}\varphi_\epsilon}(dz,ds)\nonumber\\
&&\!\!\!\!\!\!\!\!+2\int_0^t\int_Z\big\langle\big(f(s,X^{\varphi_\epsilon}(s),z)-f(s,Y^{\varphi_\epsilon}(s),z)\big)\big(\varphi_\epsilon(s,z)-1\big),X^{\varphi_\epsilon}(s)-Y^{\varphi_\epsilon}(s)\big\rangle_{F^*_{1,2}}\nu(dz)ds\nonumber\\
&&\!\!\!\!\!\!\!\!+\epsilon^2\int_0^t\int_Z\|f(s,X^{\varphi_\epsilon}(s-),z)\|^2_{F^*_{1,2}}N^{\epsilon^{-1}\varphi_\epsilon}(dz,ds).
\end{eqnarray}
In a manner similar to the method by which we got \eqref{5.01} and \eqref{5.02}, we know that the second term in the left hand-side of \eqref{LDP3} is
\begin{eqnarray}\label{LDP4}
\geq&&\!\!\!\!\!\!\!\!2\tilde{\alpha}\int_0^t|\Psi(X^{\varphi_\epsilon}(s))-\Psi(Y^{\varphi_\epsilon}(s))|_2^2ds,
\end{eqnarray}
and the first term in the right hand-side of \eqref{LDP3} is dominated by
\begin{eqnarray}\label{LDP5}
\leq&&\!\!\!\!\!\!\!\!2\tilde{\alpha}\int_0^t|\Psi(X^{\varphi_\epsilon}(s))-\Psi(Y^{\varphi_\epsilon}(s))|_2^2ds+\frac{1}{2\tilde{\alpha}}\int_0^t\|X^{\varphi_\epsilon}(s)-Y^{\varphi_\epsilon}(s)\|^2_{F^*_{1,2}}ds.
\end{eqnarray}

For the third term in the right hand-side of \eqref{LDP3}, by \textbf{(H2)(i)}, we get
\begin{eqnarray}\label{LDP9}
&&\!\!\!\!\!\!\!\!2\int_0^t\int_Z\big\langle\big(f(s,X^{\varphi_\epsilon}(s),z)-f(s,Y^{\varphi_\epsilon}(s),z)\big)\big(\varphi_\epsilon(s,z)-1\big),X^{\varphi_\epsilon}(s)-Y^{\varphi_\epsilon}(s)\big\rangle_{F^*_{1,2}}\nu(dz)ds\nonumber\\
\leq&&\!\!\!\!\!\!\!\!2\int_0^t\int_Zl_1(s,z)|\varphi_\epsilon(s,z)-1|\cdot\|X^{\varphi_\epsilon}(s)-Y^{\varphi_\epsilon}(s)\|^2_{F^*_{1,2}}\nu(dz)ds\nonumber\\
:=&&\!\!\!\!\!\!\!\!2\int_0^th_{1,\epsilon}(s)\|X^{\varphi_\epsilon}(s)-Y^{\varphi_\epsilon}(s)\|^2_{F^*_{1,2}}ds,
\end{eqnarray}
where $h_{1,\epsilon}(s)=\int_Zl_1(s,z)|\varphi_\epsilon(s,z)-1|\nu(dz)$, and from \eqref{h0}, we know that
\begin{eqnarray}\label{LDP10}
\int_0^Th_{1,\epsilon}(s)ds\leq C_{l_1,N}<\infty.
\end{eqnarray}

Combining \eqref{LDP3}-\eqref{LDP9}, we obtain that
\begin{eqnarray*}
&&\!\!\!\!\!\!\!\!\|X^{\varphi_\epsilon}(t)-Y^{\varphi_\epsilon}(t)\|^2_{F^*_{1,2}}\nonumber\\
\leq&&\!\!\!\!\!\!\!\!\int_0^t(\frac{1}{2\tilde{\alpha}}+2h_{1,\epsilon}(s))\|X^{\varphi_\epsilon}(s)-Y^{\varphi_\epsilon}(s)\|^2_{F^*_{1,2}}ds\nonumber\\
&&+2\epsilon\int_0^t\int_Z\big\langle f(s,X^{\varphi_\epsilon}(s-),z),X^{\varphi_\epsilon}(s-)-Y^{\varphi_\epsilon}(s-)\big\rangle_{F^*_{1,2}}\widetilde{N}^{\epsilon^{-1}\varphi_\epsilon}(dz,ds)\nonumber\\
&&+
\epsilon^2\int_0^t\int_Z\|f(s,X^{\varphi_\epsilon}(s-),z)\|^2_{F^*_{1,2}}N^{\epsilon^{-1}\varphi_\epsilon}(dz,ds).
\end{eqnarray*}
By Gronwall's lemma and \eqref{LDP10},
\begin{eqnarray}\label{eq Zhai 2}
&&\!\!\!\!\!\!\!\!\Bbb{E}\Big[\sup_{s\in[0,T]}\|X^{\varphi_\epsilon}(s)-Y^{\varphi_\epsilon}(s)\|^2_{F^*_{1,2}}\Big]\nonumber\\
\leq&&\!\!\!\!\!\!\!\!e^{\frac{T}{2\tilde{\alpha}}+2C_{l_1,N}}\nonumber\\
&&\cdot
\Big[
\Bbb{E}\Big(\sup_{t\in[0,T]}\Big|2\epsilon\int_0^t\int_Z\big\langle f(s,X^{\varphi_\epsilon}(s-),z),X^{\varphi_\epsilon}(s-)-Y^{\varphi_\epsilon}(s-)\big\rangle_{F^*_{1,2}}\widetilde{N}^{\epsilon^{-1}\varphi_\epsilon}(dz,ds)\Big|\Big)\nonumber\\
&&+
\Bbb{E}\Big(\epsilon^2\int_0^T\int_Z\|f(s,X^{\varphi_\epsilon}(s-),z)\|^2_{F^*_{1,2}}N^{\epsilon^{-1}\varphi_\epsilon}(dz,ds)\Big)
\Big].
\end{eqnarray}

Using the Burkhold-Davis-Gundy (BDG) inequality (with $p=1$, see \cite[Proposition 2.2]{WZ21}), \textbf{(H2)(ii)}, and Young's inequality, we obtain that
\begin{eqnarray}\label{LDP6}
&&\!\!\!\!\!\!\!\!2\epsilon\Bbb{E}\Big[\sup_{t\in[0,T]}\Big|\int_0^t\int_Z\big\langle f(s,X^{\varphi_\epsilon}(s-),z),X^{\varphi_\epsilon}(s-)-Y^{\varphi_\epsilon}(s-)\big\rangle_{F^*_{1,2}}\widetilde{N}^{\epsilon^{-1}\varphi_\epsilon}(dz,ds)\Big|\Big]\nonumber\\
\leq&&\!\!\!\!\!\!\!\!2\epsilon\Bbb{E}\Big[\int_0^T\int_Z\big\langle f(s,X^{\varphi_\epsilon}(s-),z),X^{\varphi_\epsilon}(s-)-Y^{\varphi_\epsilon}(s-)\big\rangle^2_{F^*_{1,2}}N^{\epsilon^{-1}\varphi_\epsilon}(dz,ds)\Big]^{\frac{1}{2}}\nonumber\\
\leq&&\!\!\!\!\!\!\!\!2\epsilon\Bbb{E}\Big[\int_0^T\int_Zl_2^2(s,z)\big(\|X^{\varphi_\epsilon}(s-)\|_{F^*_{1,2}}+1\big)^2\cdot\|X^{\varphi_\epsilon}(s-)-Y^{\varphi_\epsilon}(s-)\|^2_{F^*_{1,2}}N^{\epsilon^{-1}\varphi_\epsilon}(dz,ds)\Big]^{\frac{1}{2}}\nonumber\\
\leq&&\!\!\!\!\!\!\!\!2\epsilon\Bbb{E}\Big[\int_0^T\int_Z2l_2^2(s,z)\big(\|X^{\varphi_\epsilon}(s-)\|^2_{F^*_{1,2}}+1\big)\cdot\|X^{\varphi_\epsilon}(s-)-Y^{\varphi_\epsilon}(s-)\|^2_{F^*_{1,2}}N^{\epsilon^{-1}\varphi_\epsilon}(dz,ds)\Big]^{\frac{1}{2}}\nonumber\\
\leq&&\!\!\!\!\!\!\!\!2\Bbb{E}\Big[\epsilon^{\frac{1}{2}}\sup_{s\in[0,T]}\|X^{\varphi_\epsilon}(s)-Y^{\varphi_\epsilon}(s)\|^2_{F^*_{1,2}}\cdot\epsilon^{\frac{3}{2}}\int_0^T\int_Z2l_2^2(s,z)\big(\|X^{\varphi_\epsilon}(s-)\|^2_{F^*_{1,2}}+1\big)N^{\epsilon^{-1}\varphi_\epsilon}(dz,ds)\Big]^{\frac{1}{2}}\nonumber\\
\leq&&\!\!\!\!\!\!\!\!\epsilon^{\frac{1}{2}}\Bbb{E}\Big[\sup_{s\in[0,T]}\|X^{\varphi_\epsilon}(s)-Y^{\varphi_\epsilon}(s)\|^2_{F^*_{1,2}}\Big]+\epsilon^{\frac{3}{2}}\Bbb{E}\Big[\int_0^T\int_Zl_2^2(s,z)\big(\|X^{\varphi_\epsilon}(s-)\|^2_{F^*_{1,2}}+1\big)N^{\epsilon^{-1}\varphi_\epsilon}(dz,ds)\Big]\nonumber\\
\leq&&\!\!\!\!\!\!\!\!\sqrt{\epsilon}\Bbb{E}\Big[\sup_{s\in[0,T]}\|X^{\varphi_\epsilon}(s)-Y^{\varphi_\epsilon}(s)\|^2_{F^*_{1,2}}\Big]\nonumber\\
&&\!\!\!\!\!\!\!\!+\sqrt{\epsilon}\Bbb{E}\Big[\Big(\sup_{s\in[0,T]}\|X^{\varphi_\epsilon}(s)\|^2_{F^*_{1,2}}+1\Big)\int_0^T\int_Zl_2^2(s,z)\varphi_\epsilon(s,z)\nu(dz)ds\Big],
\end{eqnarray}
From \eqref{eq 5} we know that, there exists a constant $C_{l_2,2,N}$ such that
\begin{eqnarray}\label{LDP7}
C_{l_2,2,N}:=\sup_{g\in S^N}\int_0^T\int_Zl_2^2(s,z)\big(g(s,z)+1\big)\nu(dz)ds<\infty,
\end{eqnarray}
so \eqref{LDP6} is dominated by
\begin{eqnarray}\label{LDP8}
\leq\sqrt{\epsilon}\Bbb{E}\Big[\sup_{s\in[0,T]}\|X^{\varphi_\epsilon}(s)-Y^{\varphi_\epsilon}(s)\|^2_{F^*_{1,2}}\Big]+\sqrt{\epsilon} C_{l_2,2,N}\Bbb{E}\Big[\sup_{s\in[0,T]}\|X^{\varphi_\epsilon}(s)\|^2_{F^*_{1,2}}+1\Big].
\end{eqnarray}

By \textbf{(H2)(ii)} and \eqref{LDP7}, we get
\begin{eqnarray}\label{LDP11}
&&\!\!\!\!\!\!\!\!\Bbb{E}\Big[\epsilon^2\int_0^T\int_Z\|f(s,X^{\varphi_\epsilon}(s-),z)\|^2_{F^*_{1,2}}N^{\epsilon^{-1}\varphi_\epsilon}(dz,ds)\Big]\nonumber\\
=&&\!\!\!\!\!\!\!\!\epsilon\Bbb{E}\Big[\int_0^T\int_Z\|f(s,X^{\varphi_\epsilon}(s),z)\|^2_{F^*_{1,2}}\varphi_\epsilon(s,z)\nu(dz)ds\Big]\nonumber\\
\leq&&\!\!\!\!\!\!\!\!2\epsilon\Bbb{E}\Big[\int_0^T\int_Z\big(\|X^{\varphi_\epsilon}(s)\|^2_{F^*_{1,2}}+1\big)l_2^2(s,z)\varphi_\epsilon(s,z)\nu(dz)ds\Big]\nonumber\\
\leq&&\!\!\!\!\!\!\!\!2\epsilon C_{l_2,2,N}\Bbb{E}\Big[\sup_{s\in[0,T]}\|X^{\varphi_\epsilon}(s)\|^2_{F^*_{1,2}}+1\Big].
\end{eqnarray}

Combining \eqref{eq Zhai 2}-\eqref{LDP11}, we get, for any $\epsilon\in(0,\frac{1}{4}e^{-\frac{T}{\tilde{\alpha}}-4C_{l_1,N}}]$,
\begin{eqnarray}\label{LDP12}
&&\!\!\!\!\!\!\!\!\Bbb{E}\Big[\sup_{s\in[0,T]}\|X^{\varphi_\epsilon}(s)-Y^{\varphi_\epsilon}(s)\|^2_{F^*_{1,2}}\Big]\nonumber\\
\leq&&\!\!\!\!\!\!\!\!6\sqrt{\epsilon} C_{l_2,2,N}\Bbb{E}\Big[\sup_{s\in[0,T]}\|X^{\varphi_\epsilon}(s)\|^2_{F^*_{1,2}}+1\Big]\cdot e^{\frac{T}{2\tilde{\alpha}}+2C_{l_1,N}}\nonumber\\
:=&&\!\!\!\!\!\!\!\!\sqrt{\epsilon }C_{T,l_1,l_2,2,\tilde{\alpha},N}\Bbb{E}\Big[\sup_{s\in[0,T]}\|X^{\varphi_\epsilon}(s)\|^2_{F^*_{1,2}}+1\Big].
\end{eqnarray}
Applying It\^{o}'s formula to $\|X^{\varphi_\epsilon}(s)\|^2_{F^*_{1,2}}$ and using arguments similar to how we got \eqref{LDP12}, we can get
\begin{eqnarray}\label{LDP100}
&&\!\!\!\!\!\!\!\!\Bbb{E}\Big[\sup_{s\in[0,T]}\|X^{\varphi_\epsilon}(s)\|^2_{F^*_{1,2}}\Big]\nonumber\\
\leq&&\!\!\!\!\!\!\!\! \widetilde{C}_{T,l_2,2,\tilde{\alpha},N}|x|_2^2
+\sqrt{\epsilon }\widetilde{C}_{T,l_2,2,\tilde{\alpha},N}\Bbb{E}\Big[\sup_{s\in[0,T]}\|X^{\varphi_\epsilon}(s)\|^2_{F^*_{1,2}}+1\Big].
\end{eqnarray}
To prove (\ref{LDP100}), $\Psi(0)=0$(see \textbf{(H1)}) and \textbf{(H2)(i)} will be used.

Here the constant $\widetilde{C}_{T,l_2,2,\tilde{\alpha},N}$ is independent of $\epsilon$. The inequality above implies that there exists $\epsilon_0>0$ small enough such that
\begin{eqnarray}\label{LDP13}
\sup_{\epsilon\in(0,\epsilon_0)}\Bbb{E}\Big[\sup_{s\in[0,T]}\|X^{\varphi_\epsilon}(s)\|^2_{F^*_{1,2}}\Big]\leq C<\infty.
\end{eqnarray}

Combining \eqref{LDP13} with \eqref{LDP12}, we know that
\begin{eqnarray*}
X^{\varphi_\epsilon}\longrightarrow Y^{\varphi_\epsilon}~~\text{in}~~L^2(\Omega;L^\infty([0,T];F^*_{1,2}))~~\text{as}~~\epsilon\rightarrow0,
\end{eqnarray*}
which implies (b) in Condition \ref{Condition}.

\end{proof}

\subsection*{Acknowledgements}
The authors would like to thank the referees for their valuable suggestions which helped to improve the paper.

\end{document}